\documentclass[12pt,reqno]{amsart}
\usepackage[top=1in,bottom=1in,left=1.1in,right=1.1in]{geometry}
\usepackage{color}

\usepackage{amssymb}
\usepackage{amsmath}
\usepackage{amsfonts}
\usepackage{geometry}
\usepackage{graphicx}
\usepackage{mathrsfs,amssymb}

\usepackage{hyperref}
\usepackage{cleveref}

\usepackage{enumerate}

 \newtheorem{theorem}{Theorem}[section]
 \newtheorem{proposition}[theorem]{Proposition}
 
 \newtheorem{lemma}[theorem]{Lemma}
 \newtheorem{corollary}[theorem]{Corollary}

 \theoremstyle{definition}
 \newtheorem{definition}[theorem]{Definition}

 \theoremstyle{remark}
 \newtheorem{remark}[theorem]{Remark}

\numberwithin{equation}{section}

\renewcommand{\Re}{\operatorname{Re}}

\newcommand{\rg}{\operatorname{rg}}

\newcommand{\R}{\mathbb{R}}
\newcommand{\la}{\lambda}

\newcommand{\C}{\mathbb{C}}
\newcommand{\Hb}{\overline{\mathbb{H}}}
\newcommand{\mb}[1]{\textbf{#1}}

\newcommand{\mc}[1]{\mathcal{#1}}

\newcommand{\B}{\mathbb{B}}

\title[]{Stable blowup for the supercritical hyperbolic Yang-Mills equations}

\author{Irfan Glogi\'c}
\address{Department of Mathematics, University of Vienna, Oskar-Morgenstern-Platz 1, 1090 Vienna, Austria}
\email{irfan.glogic@univie.ac.at}

\thanks{The author acknowledges support by the Austrian Science Fund FWF, Projects P 30076 and P 34378.}

\usepackage{kantlipsum}

\begin{document}

\begin{abstract}
We consider the Yang-Mills equations in $(1+d)$-dimensional Minkowski spacetime. It is known that in the supercritical case, i.e., for $d \geq 5$, these equations admit closed form equivariant self-similar blowup solutions \cite{BieBiz15}. These solutions are furthermore conjectured to be the universal attractors for generic large equivariant data evolutions. In this paper we partially prove this conjecture. Namely, we show that for all odd $d \geq 5$ the blowup mechanism exhibited by these solutions is stable.
\end{abstract}

\maketitle
\section{Introduction}
\noindent We consider the connection $1$-forms $A=(A_0,A_1,\dots,A_d)$
with\footnote{Here, and throughout the paper, Greek indices run from $0$ to $d$, while Latin indices go from $1$ to $d$, and Einstein summation convention is in force. Also, the indices are raised and lowered with respect to the Minkowski metric $\eta_{\alpha\beta}:=\text{diag} (-1,1,\dots,1)$.} $A_\alpha: \R^{1+d} \mapsto \mathfrak{so}(d,\R)$, where $\mathfrak{so}(d,\R)$ is the Lie algebra of the Lie group $SO(d,\R)$, i.e., its elements are real skew-symmetric $(d\times d)$-matrices endowed with the commutator bracket.  Furthermore, the associated covariant derivative acting on $\mathfrak{so}(d,\R)$-valued functions is defined by
\begin{equation}
	{\bf D}_\alpha:=\partial_\alpha + [A_\alpha,\cdot],
\end{equation}
and the curvature tensor $F=F[A]$ is given by
\begin{equation}
	F_{\alpha\beta}= \partial_\alpha A_\beta - \partial_\beta A_\alpha + [A_\alpha,A_\beta].
\end{equation}
With this, we define the so-called Yang-Mills action functional
\begin{equation*}
	\mathcal{S}[A]:=\int_{\R^{1+d}} \text{tr} (F_{\alpha\beta}F^{\alpha\beta}). 
\end{equation*}
Subsequently, the associated Euler-Lagrange equations read
\begin{equation}\label{Eq:YM_general}
	\mathbf{D}^\alpha F_{\alpha\beta}=0,
\end{equation}
and are called the \emph{hyperbolic Yang-Mills equations}. Furthermore, solutions to Eq.~\eqref{Eq:YM_general} are called Yang-Mills connections on the trivial bundle $\R^{1+d} \times SO(d,\R)$.
Due to the Lorentzian nature of the Minkowski space, \eqref{Eq:YM_general} represents an evolution system in the zeroth slot variable $x^0$, which we, according to custom, denote by $t$. Consequently, the initial data consist of pairs $(\tilde{A},\tilde{B})$ of $\mathfrak{so}(d,\R)$-valued 1-forms on $\R^d$. This system is, however, under-determined due to the freedom of the choice of \emph{gauge}. To remove this ambiguity, we study the Cauchy problem for Eq.~\eqref{Eq:YM_general} in the so-called \emph{temporal gauge}, i.e., we look for solutions for which $A_0 \equiv 0$. And accordingly, we pose initial conditions
\begin{equation}\label{Eq:YM_gen_init_cond}
	A(0,\cdot)=\tilde{A}, \quad  \partial_0 A(0,\cdot)=\tilde{B}.
\end{equation}

Since $SO(d,\R)$ is noncommutative, the system \eqref{Eq:YM_general} is nonlinear, and this brings about an important question, namely that of \emph{regularity} versus \emph{breakdown}. Is it possible that perfectly smooth and localized initial data \eqref{Eq:YM_gen_init_cond} become singular in finite time? In this context, the relationship between the nonlinear scaling and the conserved energy leads to useful heuristics about the existence of blowup. Namely, the system \eqref{Eq:YM_general} is invariant under the rescaling $A \mapsto A_{\lambda}$, where
\begin{equation*}
	A_{\lambda,\alpha}(t,x):= \la^{-1} A_{\alpha}(t/\la,x/\la), \quad \la>0,
\end{equation*}
and the conserved (and positive definite) energy
\begin{equation*}
	\mathcal E[A](t):= -\sum_{0 \leq \alpha < \beta \leq d} \int_{\R^d}  \text{tr} ( F^2_{\alpha\beta}(t,x) )\,  dx,
\end{equation*}
scales like $\mathcal E[A_\la](t)=\la^{d-4} \mathcal E[A](t/\la)$. Consequently, the Yang-Mills equations \eqref{Eq:YM_general} are \emph{energy-subcritical} if $d \leq 3$, \emph{energy-critical} if $d=4$, and \emph{energy-supercritical} if $d \geq 5$.
The heuristic we mentioned above says that (even in the more general setting of classical field theories) subcritical equations are globally regular, but for supercritical ones, although solutions with sufficiently small initial data exist globally, large data solutions blowup in finite time, see Klainerman \cite{Kla97}.

The Yang-Mills system \eqref{Eq:YM_general} does, in fact, obey this heuristic in the physical dimension $d=3$. Namely, from the early works of Eardley and Moncrief \cite{EarMon82a,EarMon82b}, it is known that in this setting global regularity holds. On the other hand, in both the critical and supercritical case blowup is possible. This can be seen already in a subclass of the so-called \emph{equivariant solutions}. These are temporal gauge solutions that have the following form 
	 \begin{equation}\label{Def:Equiv_ansatz}
	 A^{ij}_k(t,x)=(\delta^j_k x^i - \delta^i_k x^j)u(t,|x|),
	 \end{equation}
	 for some $u(t,\cdot):[0,\infty) \rightarrow \R$, which we call the \emph{radial profile} at time $t$ of the equivariant 1-form $A$. The ansatz \eqref{Def:Equiv_ansatz} reduces the system \eqref{Eq:YM_general} to a $(d+2)$-dimensional semilinear radial wave equation for the profile $u$ (see \cite{Dum82} for derivation)
	 \begin{equation}\label{Eq:YM_equiv}
	 \displaystyle{\left( \partial_t^2 - \partial_r^2 - \frac{d+1}{r}\partial_r \right) u(t,r) =(d-2)u(t,r)^2 \big(3-r^2u(t,r)\big)},
	 \end{equation}	 
with the corresponding initial condition 
\begin{equation}\label{Eq:YM_equiv_init_cond}
	u(0,r)=u_0(r), \quad \partial_0u(0,r)=u_1(r),
\end{equation}
which is obtained from the equivariant initial data
\begin{equation}\label{Def:Equiv_init_data}
	\tilde{A}_k^{ij}(x) = (\delta^j_k x^i - \delta^i_k x^j)u_0(|x|), \quad \tilde{B}_k^{ij}(x) = (\delta^j_k x^i - \delta^i_k x^j)u_1(|x|).
\end{equation}

In the energy-critical case, $d=4$, Eq.~\eqref{Eq:YM_equiv} was numerically studied by Bizo\'n and Tabor \cite{BizTab01}, who observed singularity formation at the origin for generic large initial data. This was then followed by a rigorous proof of the  existence of blowup by Krieger, Schlag and Tataru \cite{KriSchTat09}. Later on, remarkably, Rapha\"el and Rodnianski \cite{RapRod12} constructed blowup that is stable.
	 
For the supercritical case, $d \geq 5$, the existence of blowup for Eq.~\eqref{Eq:YM_equiv} was first shown by Cazenave, Shatah and Tahvildar-Zadeh \cite{CazShaTah98}, who constructed self-similar solutions for $d \in \{ 5,7,9 \} $. They then lifted these solutions to one dimension higher to construct (non-equivariant) singular traveling waves in the temporal gauge, thereby exhibiting blowup for \eqref{Eq:YM_general} in the three lowest even supercritical dimensions as well. They, however, did not investigate the role of these solutions for generic evolutions. This issue was later approached from the numerical point of view in the aforementioned paper by Bizo\'n and Tabor \cite{BizTab01}. Their simulations show evidence that in the lowest supercritical dimension, $d=5$, the self-similar solution constructed by Cazenave et al.~is in fact a generic blowup profile. Later on, Bizo\'n \cite{Biz02} found this solution in closed form
\begin{equation*}
u_{0,T}(t,r):=\frac{1}{(T-t)^2}\frac{8}{3\rho^2+5}, \quad \rho=\frac{r}{T-t}, \quad T>0,
\end{equation*} 
and proved that in addition to it, there is a countable family of smooth self-similar profiles; this situation indicates the richness of blowup dynamics already within the equivariance class.

 A rigorous proof of the stability of $u_{0,T}$, however, came about more recently in the works of Costin, Donninger, Huang and the author \cite{Don14a,CosDonGloHua16}. In the meantime, Biernat and Bizo\'n \cite{BieBiz15} discovered the analogue of $u_{0,T}$ in all supercritical dimensions
\begin{equation}\label{Def:BB_sol}
u_T(t,r):=\frac{1}{(T-t)^2}\frac{\alpha(d)}{\rho^2+\beta(d)}, \quad \rho = \frac{r}{T-t}, \quad T>0,
\end{equation}
where 
\begin{equation}\label{Def:Alpha_Beta}
\alpha(d)=2\left(1+\sqrt{\tfrac{d-4}{3(d-2)}} \right) \quad \text{and} \quad \beta(d)=\frac{1}{3}\big(2d-8+\sqrt{3(d-2)(d-4)}\big).
\end{equation}
This, via \eqref{Def:Equiv_ansatz}, leads to a temporal gauge solution $A_T$ of the system \eqref{Eq:YM_general}
\begin{equation}\label{Def:BB_sol_vector}
(A_T(t,x))^{ij}_k:=(\delta^j_k x^i - \delta^i_k x^j)u_T(t,|x|),
\end{equation}
which suffers gradient blowup at the origin as $t \rightarrow T^-$.
Furthermore, based on numerical simulations, Biernat and Bizo\'n conjectured in \cite{BieBiz15} that $A_T$ is the generic blowup profile within the equivariant solution class. The main result of this paper provides the first step in proving this conjecture. More precisely, we show that for all odd dimensions $d \geq 5$ there is an open set of equivariant initial data around $(A_1(0,\cdot), \partial_0 A_1(0,\cdot))$ for which the Cauchy evolution of \eqref{Eq:YM_general}-\eqref{Eq:YM_gen_init_cond} blows up by converging back to $A_T$ for some $T$ close to 1. For the formal statement of the main result we use the following notation $$A[t]:=(A(t,\cdot),\partial_tA(t,\cdot)),$$ and for precise definitions of the solution concept and the blowup time we refer the reader to Sec.~\ref{Subsec:Well-posedness}.

\begin{theorem}\label{Thm:Main}
	Let $d \geq 5$ be odd. Then there exist $M,\varepsilon,\omega>0$ such that for any equivariant initial data $A[0]$ for which
	\begin{equation*}
		\| A[0]- A_1[0] \|_{H^\frac{d+1}{2} \times H^\frac{d-1}{2}(\mathbb{B}^d_{1+\varepsilon})} < \frac{\varepsilon}{M},
	\end{equation*}
	the corresponding time evolution for the Cauchy problem \eqref{Eq:YM_general}-\eqref{Eq:YM_gen_init_cond} forms a singularity at the origin in finite time and the following hold:
    \begin{itemize}
    	\item[(1)] the blowup time $T$ belongs to the interval $[1-\varepsilon,1+\varepsilon]$,
    	\item[(2)] the corresponding solution $A[t]$ satisfies
    	\begin{equation}\label{Eq:MainThmEst}
    		\frac{\| A[t]- A_T[t] \|_{\dot{H}^k \times \dot{H}^{k-1}(\mathbb{B}^d_{T-t})}}{\| A_T[t] \|_{\dot{H}^k \times \dot{H}^{k-1}(\mathbb{B}^d_{T-t})}} \lesssim   (T-t)^{\omega},
    	\end{equation}
    	for $k = 1,2, \dots ,(d+1)/2 $ and all $t \in [0,T)$. In particular,
    	\begin{equation}\label{Eq:MainThmEst_fullnorm}
    		\frac{\| A[t]- A_T[t] \|_{H^\frac{d+1}{2} \times H^\frac{d-1}{2}(\mathbb{B}^d_{T-t})}}{\| A_T[t] \|_{H^\frac{d+1}{2} \times H^\frac{d-1}{2}(\mathbb{B}^d_{T-t})}} \lesssim  (T-t)^{\omega},
    	\end{equation}
    	for all $t \in [0,T)$.
    \end{itemize}
	
\end{theorem}
\begin{remark}
	Note that simply by scaling we get 
	\begin{equation}\label{Eq:BlowupEst}
		\| \emph{A}_T[t] \|_{\dot{H}^k \times \dot{H}^{k-1}(\mathbb{B}^d_{T-t})} \simeq (T-t)^{\frac{d}{2}-1-k} \quad \text{as} \quad t \rightarrow T^-.
	\end{equation}
 Hence, blowup of $A_T[t]$ is detected by Sobolev norms of degree larger than $d/2-1$. The estimates \eqref{Eq:MainThmEst}-\eqref{Eq:MainThmEst_fullnorm} in particular imply that the solution $A[t]$ obeys the same leading asymptotic as $A_T[t]$, and although the difference $A[t]-A_T[t]$ might blowup as well, it does so at a slower rate. 
\end{remark}

\begin{remark}
	To avoid unnecessary technicalities, we restrict our stability analysis to $A_1$, as the general case follows by an analogous process, which furthermore does not bring any new insight.
\end{remark}

\begin{remark}
	For $d=5$, a version of Theorem \ref{Thm:Main} was established by Donninger in \cite{Don14a}. The differences are that his stability result concerns the radial profile $u_T$, the underlying topology is not given by Sobolev norms (although it is equivalent to the one we work in), and the result is conditional on a certain spectral assumption, which was, however, subsequently resolved in \cite{CosDonGloHua16}. Theorem \ref{Thm:Main} can be therefore interpreted as the extension of the works \cite{Don14a,CosDonGloHua16} to all higher odd dimensions. 
\end{remark}

\subsection{Equivalence of Cauchy problems}
 As it is apparent from the statement of Theorem \ref{Thm:Main}, we study the Cauchy  evolution of \eqref{Eq:YM_general}-\eqref{Eq:YM_gen_init_cond} inside lightcones with equivariant initial data prescribed on balls centered at the origin. In addition, since the flow preserves equivariance symmetry, the 1-form solution is determined by the evolution of its radial profile. We therefore proceed with stating an important, and for our approach relevant, relation between 1-forms and their radial profiles.
\begin{proposition}\label{Prop:Equiv_1forms}
	 Let $d \geq 2$, $k \in \mathbb{N}_0$ and $R>0$. If $k <  \frac{d}{2}+2$ then for an equivariant 1-form $A$ on $\B^d_R $ with the radial profile $u$, i.e., for
	 \[ 
	 {A}^{ij}_k(x)=(\delta^j_k x^i - \delta^i_k x^j)u(|x|), \quad x \in \B^d_R, \quad  i,j,k \in \{1,\dots,d\},
	 \]
	we have that ${A} \in H^k(\B_R^d)$ if and only if $u(|\cdot|) \in H^k(\B_R^{d+2})$. What is more,
	\begin{equation}\label{Eq:Equiv_1form}
		\| {A} \|_{H^k(\mathbb{B}^d_R)} \simeq \| u(|\cdot|) \|_{H^k(\mathbb{B}^{d+2}_R)}, 
	\end{equation}
	for all equivariant 1-forms ${A} \in H^k(\mathbb{B}^d_R)$, where the implicit constants in \eqref{Eq:Equiv_1form} can be chosen uniformly for all $R$ in a bounded set positively distanced from zero.
\end{proposition}

  To the knowledge of the author, equivalence results of this type have not appeared in the literature. For that reason, we devote  Appendix \ref{App:SobolevCorot} to the proof of Proposition \ref{Prop:Equiv_1forms} and furthermore, as a by-product we obtain analogous results for other geometrically relevant maps, e.g., the corotational ones.

 Now, due to equivalence \eqref{Eq:Equiv_1form},  we work exclusively on the Cauchy problem \eqref{Eq:YM_equiv}-\eqref{Eq:YM_equiv_init_cond}. Namely, the bulk of our work will consist of obtaining a stability result for $u_T$, which we then via \eqref{Eq:Equiv_1form} reformulate in terms of the 1-form it generates and thereby obtain Theorem \ref{Thm:Main}.

\subsection{History of the problem and related results and models}
The Yang-Mills equations came about as a model of particle physics, and are in fact the basic equations of
gauge theories describing the fundamental forces of nature, see \cite{Tho05,Ham17}. On the other hand, the classical field theory interpretation of the Yang-Mills model  has been of significant interest as well, see  \cite{Act79}.
Furthermore, understanding the problem of singularity formation for Yang-Mills equations is expected to shed light onto understanding the problem of collapse for more complicated models from general relativity, see \cite{BizTab01,GunMar07}.
Although the physical importance of understanding their solutions can not be overstated, the Yang-Mills equations are very much interesting from the pure mathematical point of view as well, and so far the bulk of the literature focuses on the subcritical and the critical case. 

For $d=3$, one of the earliest notable works on the problem is by Choquet-Bruhat and Christodoulou \cite{ChoChr81}, who proved a small data global existence result. This was followed by Eardley and Moncrief \cite{EarMon82b}, who showed existence of global classical solutions for smooth enough data of any size. Later on, Klainerman and Machedon \cite{KlaMac95} established well-posedness in the energy space, and then used conservation of energy to obtain global existence. This illustrates one of the dominant themes in the nonlinear wave equation theory, namely that of proving well-posedness in as weak spaces as possible, and much of the subsequent work on the problem is in that spirit. An alternative proof of well-posedness in the energy norm was more recently given by Oh \cite{Oh14}, which he then subsequently upgraded into a global existence result \cite{Oh15}. Well-posedness for small data below the energy norm was proved by Tao \cite{Tao03}, while the smallness assumption was removed recently by Oh and Tataru \cite{OhTat19}. Also, for some recent results that concern well-posedness in the Lorenz gauge, see \cite{SelAch16,Ach15a,Ach15b}.

For $d \geq 4$, the local well-posedness at an almost-optimal regularity in the Coulomb gauge was established by Klainerman and Tataru \cite{KlaTat99}. Recently, Oh and Tataru \cite{OhTat19} proved local well-posedness at the optimal regularity for the temporal gauge. Somewhat before that, Krieger and Tataru \cite{KriTat17} showed global existence for data with small energy.  As opposed to the case $d=3$, in $d=4$ there is a static equivariant solution called the \emph{instanton}. The energy of the instanton plays a role of the threshold for global existence; Cot\^e, Kenig and Merle \cite{CotKenMer08} proved global existence and scattering for equivariant initial data with
energy below or equal to the energy of the instanton. This line of research was recently significantly extended in a remarkable sequence of works by Oh and Tataru \cite{OhTat17arx,OhTat20,OhTat19,OhTat21,OhTat19b}, where they establish a threshold conjecture and the soliton-bubbling versus scattering dichotomy.

\subsubsection{Supercritical corotational wave maps}
The Yang-Mills system \eqref{Eq:YM_general} under equivariance symmetry (i.e., Eq.~\eqref{Eq:YM_equiv}) bears many similarities to corotational wave maps equation for maps between the Minkowski space $\R^{1+(d-2)}$ and rotationally symmetric warped product manifolds $\R^{+} \times_g \mathbb
S^{d-3}$,  
\begin{equation}\label{Eq:WM_equiv}
	\left( \partial_t^2 - \partial_r^2 - \frac{d-1}{r}\partial_r \right) u(t,r) =\frac{d-3}{r^3}\Big(ru(t,r)-g\big(ru(t,r)
	\big)g'\big(ru(t,r)\big)\Big),
\end{equation}
where $g$ is the warping function for the target manifold.
The parallel between Eq.~\eqref{Eq:WM_equiv} and Eq.~\eqref{Eq:YM_equiv} in the supercritical case $d \geq 5$ is very nicely illustrated in the context of local well-posedness and self-similar blowup in a paper by Cazenave et al.~\cite{CazShaTah98}, see also \cite{ShaTah94}. A particularly prominent example of Eq.~\eqref{Eq:WM_equiv} is the one for the  sphere as the target, i.e., for $g(u)=\sin(u)$. In that case, for $d=5$, both Eqs.~\eqref{Eq:WM_equiv} and \eqref{Eq:YM_equiv} admit infinitely many self-similar solutions, as proven by Bizo\'n \cite{Biz00,Biz02}, with the ``ground state" being conjectured to drive the generic blowup. More recently, this analogy has been extended to higher dimensions by Biernat and Bizo\'n \cite{BieBiz15} (see also \cite{BieBizMal17}), who also found the closed form expression of the ground state for both equations. The ground state for Eq.~\eqref{Eq:WM_equiv} is known to be stable in all odd dimensions, as proven in a pair of papers by Chatzikaleas, Costin, Donninger, and the author \cite{CosDonGlo17,ChaDonGlo17}, see also \cite{Don11,DonSchAic12} for the first (conditional) stability results.

\subsubsection{The Yang-Mills heat flow} It is instructive to finish the literature review with a discussion on the parabolic analogue of the Yang-Mills system \eqref{Eq:YM_general}. Namely, in order to construct Yang-Mills connections on the trivial bundle $\R^{d} \times SO(d,\R)$, one defines a flow by artificially adding the time derivative of the connection 1-form. Under equivariance symmetry, this reduces to a radial heat equation in $d+2$ dimensions 
\begin{equation}\label{Eq:YMHF_equiv}
	\displaystyle{\left( \partial_t - \partial_r^2 - \frac{d+1}{r}\partial_r \right) u(t,r) =(d-2)u(t,r)^2 \big(3-r^2u(t,r)\big)}.
\end{equation}
As opposed to its hyperbolic counterpart, the Yang Mills heat flow does not admit blowup in the critical dimension $d=4$, even without symmetry assumptions, as recently proven by Waldron \cite{Wal19}. In the supercritical case, however, Naito \cite{Nai94} showed that Eq.~\eqref{Eq:YMHF_equiv} admits blowup, see also \cite{Gro01}. Later on, Gastel \cite{Gas02} constructed self-similar (blowup) solutions for $5 \leq d \leq 9$, with an explicit example subsequently given by Weinkove \cite{Wei04}. Very recently, Weinkove's solution has been proved to be stable, in $d=5$ by Donninger and Sch\"orkhuber \cite{DonSch19}, and for higher dimensions by Sch\"orkhuber and the author \cite{GloSch20}. Interestingly, as opposed to Eq.~\eqref{Eq:YM_equiv}, for $d \geq 10$, possibility of self-similar blowup for Eq.~\eqref{Eq:YMHF_equiv} is ruled out \cite{BizWas15}, and generic blowup is therefore of non-self-similar type.

\subsection{Overview of the paper and sketch of the proof of Theorem \ref{Thm:Main}} 

After establishing the equivalence result Proposition \ref{Prop:Equiv_1forms}, the rest of the paper focuses exclusively on the Cauchy problem \eqref{Eq:YM_equiv}-\eqref{Eq:YM_equiv_init_cond}.

In Sec.~\ref{Sec:Well-posedness} we develop a well-posedness theory for \eqref{Eq:YM_equiv}-\eqref{Eq:YM_equiv_init_cond} inside lightcones. For this, we utilize the similarity variables 
\begin{equation*}\label{Def:SS_var_intro}
	\tau=\tau(t):=\ln T - \ln(T-t), \quad\rho=\rho(t,r):=\frac{r}{T-t}.
\end{equation*}
In this way, the backward lightcone of the point $(T,0)$ is mapped into the infinite cylinder centered at the origin and of unit radius.
By also rescaling the dependent variables
\begin{equation*}\label{Def:scaled_u_intro}
	\psi(\tau,\rho):=(T-t)^2u(t,r) \quad \text{and} \quad \varphi(\tau,\rho):=(T-t)^3\partial_tu(t,r),
\end{equation*}
the self-similar solution $u_T$ becomes $\tau$-independent $\Psi_{\text{st}}=\Psi_{\text{st}}(\rho)$. 
What is more, along with localizing the evolution to lightcones, in this way we turn the problem of stability of finite time blowup of $u_T$ inside the backward lightcone of the blowup point, into the (generally more familiar) problem of the asymptotic stability of a static solution $\Psi_{\text{st}}$ inside an infinite cylinder. But, to carry out the stability analysis, we first need a well-posedness theory for the Cauchy evolution in the new variables. To this end, we let $\Psi(\tau):=\big(\psi(\tau,\cdot),\varphi(\tau,\cdot) \big)$,
by means of which \eqref{Eq:YM_equiv}-\eqref{Eq:YM_equiv_init_cond} adopts the following form
\begin{equation}\label{Eq:YM_system_abstract_intro}
	\begin{cases}
		\Psi'(\tau)=\widetilde{\mb L}_0\Psi(\tau)+\mb N_0(\Psi(\tau))\\	
		\Psi(0)=\mb U_0(T).
	\end{cases}
\end{equation}
Here, $ \widetilde{\mb L}_0$ is the $(d+2)$-dimensional
radial wave operator in similarity variables, $\mb N_0$ is the remaining nonlinear operator, and the initial data is $\mb U_0(T) = \big( T^2 u_0(T\cdot) , T^3u_1(T\cdot) \big).$
We show local well-posedness of \eqref{Eq:YM_system_abstract_intro} in the following space of radial Sobolev functions
\begin{equation*}\label{Def:H_intro}
	\mathcal{H}:=H_r^{\frac{d+1}{2}} \times H_r^{\frac{d-1}{2}}(\mathbb{B}^{d+2}),
\end{equation*}
see Sec.~\ref{Sec:Funct_setup} for the precise definition. To achieve this, we use the semigroup theory. Namely, we first show that for every odd $d \geq 5$, (the closure of) the operator $\widetilde{\mb L}_0$ generates a one-parameter strongly continuous and exponentially decaying semigroup $\mb S_0(\tau)$ of bounded operators on $\mc H$, see Proposition \ref{Prop:free_semigroup}. The proof is somewhat involved and entails a carefully tailored inner product on $\mc H$, which allows for the efficient use of the reduction of the linear wave equation in odd dimensions to the one in one dimension (hence our focus on odd $d$). With this at hand, we  recast the problem \eqref{Eq:YM_system_abstract_intro} in the integral form \`a la Duhamel
\begin{equation*}\label{Eq:Duhamel0_intro}
	\Psi(\tau)=\mb S_0(\tau)\mb U_0(T) + \int_{0}^{\tau}\mb S_0(\tau-s)\mb N_0(\Psi(s))ds,
\end{equation*}
wherefrom by a fixed point argument we establish the existence and uniqueness of local strong solutions to \eqref{Eq:YM_system_abstract_intro}, for arbitrary initial data in $\mc H$. By undoing the similarity transformations, this then turns into local well-posedness of \eqref{Eq:YM_equiv}-\eqref{Eq:YM_equiv_init_cond} in lightcones, see Theorem \ref{Thm:Well-posedness}.

In Sec.~\ref{Sec:Stability_analysis} we carry out the stability analysis of $\Psi_{\text{st}}$.  For this, according to custom, we consider solutions to \eqref{Eq:YM_system_abstract_intro} of the following form $\Psi(\tau)=\Psi_{\text{st}}+\Phi(\tau)$.
 Such ansatz leads to the initial value problem for $\Phi$,
\begin{equation}\label{Eq:AbsEvolPhi}
	\begin{cases}
	\Phi'(\tau)=\widetilde{\mb L}_0\Phi(\tau)+\mb{L}'\Phi(\tau)+{\mb N}(\Phi(\tau))\\
	\Phi(0)= \mb U (\mb v,T),
	\end{cases}
\end{equation}
where $\mb L'$ is the Fr\'echet derivative of $\mb N_0$ at
$\Psi_{\text{st}}$ and $\mb N$ is the remaining nonlinear operator. Furthermore, for convenience, we write $\mb U (\mb v,T)= \mb U_0(T)- \mb U_0(1) + \mb v$, where $\mb v := \mb U_0(1) - \Psi_{\text{st}}$. The operator $\mb L': \mc H \rightarrow \mc H$ is compact, and therefore the closure of the operator $\widetilde{\mb L}_0 + \mb L'$, which we denote by $\mb L$, generates another semigroup $\mb S(\tau)$ on $\mc H$, by means of which we recast \eqref{Eq:AbsEvolPhi} in the integral form
\begin{equation}\label{Eq:Duhamel1}
	\Phi(\tau)=\mathbf{S}(\tau)\mathbf{U}(\mathbf{v},T)+\int_{0}^{\tau}\mathbf{S}(\tau-s)\mathbf{N}(\Phi(s))ds.
\end{equation}
Now, asymptotic stability of $\Psi_{\text{st}}$ is equivalent to small data global existence and decay of solutions to Eq.~\eqref{Eq:Duhamel1}. Proving this, however, necessitates some sort of decay of $\mb S(\tau)$. It turns out that, to understand the growth properties of $\mb S(\tau)$, it is enough to study the spectrum of its generator $\mb L$. This is possible because of a spectral mapping property that underlies this setting. In fact, this follows from a general theorem on compactly perturbed semigroups we prove in the appendix, see Theorem \ref{Thm:Semigroups}.  Consequently, the core of this section, and of the proof of the main result overall, is the spectral analysis of $\mb L$. We note, however, that due to the highly non-self-adjoint character of the operator $\mb L$, studying its spectrum is an extremely difficult problem. In spite of this, we managed to prove that $\mb L$ has precisely one unstable eigenvalue $\la=1$, see Proposition \ref{Prop:unstable_spectrum}. This property of $\mb L$ is equivalent to the so called \emph{mode stability} of $u_T$, and for the lowest dimension $d=5$ this has been an open problem for quite a while (see  \cite{BizChm05} for the initial appearance), until it was somewhat recently resolved in a paper by Costin, Donninger, Huang and the author \cite{CosDonGloHua16}. The techniques from \cite{CosDonGloHua16} were subsequently adapted in \cite{CosDonGlo17} to all higher dimensions in the context of wave maps, and we follow this methodology here. What is more, we significantly simplify the method from \cite{CosDonGlo17}, and we make it more transparent and robust, see Remark \ref{Rem:Quasi}.

 The unstable eigenvalue $\la=1$ is, in fact, not a ``real" instability, as it is related to the time translation symmetry of the equation. To deal with this instability, we use a Lyapunov-Perron type of argument. First, we define the Riesz projection $\mb P$ associated with $\la=1$. This projection turns out to have rank one, and it furthermore yields exponential decay on the stable subspace
\begin{equation}\label{Eq:Decay}
	\|\mb S(\tau)(1 - \mb P) \|_{\mc H} \leq M e^{-\omega \tau}, \quad \omega >0.
\end{equation}
 We emphasize here that, according to our theorem on compactly perturbed semigroups (Theorem \ref{Thm:Semigroups}), the information on the unstable point spectrum is enough to conclude the spectral gap property of $\mb L$, which then implies the exponential decay of $\mb S(\tau)$ on the stable subspace. This, most importantly, avoids having to estimate the resolvent of $\mb L$ (in order to invoke Gearhart-Pr\"uss), which was done in all of the implementations of this approach so far (in case the entire point spectrum could not be computed explicitly). 
Then, we suppress the symmetry-induced unstable direction of $\mb S(\tau)$
by introducing a correction term
\[ \mb C(\Phi,\mb U(\mb v,T)):=\mb P\left(\mb U(\mb
v,T)+\int_{0}^{\infty}e^{-s}\mb N(\Phi(s))ds \right) \]
into Eq.~\eqref{Eq:Duhamel1}, i.e., we consider the
modified equation
\begin{align}\label{Eq:Modified1}
	\Phi(\tau)=\mb S(\tau)\big(\mb U(\mb v,T)-\mb C(\Phi,\mb U(\mb v,T))\big)  + \int_0^{\tau} \mb S(\tau - s)  \mb N(\Phi(s)) ds.
\end{align}
We then prove by a fixed point argument that for small enough initial
data $\mb v$, every $T$ close to $1$ yields a unique solution to
Eq.~\eqref{Eq:Modified1} that decays to zero at the
rate given in \eqref{Eq:Decay}. 
Finally,  we use the presence of the time-translation-symmetry eigenvalue
$\la=1$ to single out a particular $T$ near $1$ for
which $\mb C(\Phi,\mb U(\mb v,T))=0$, and thereby obtain global decaying
solutions to Eq.~\eqref{Eq:Duhamel1}. This, when translated back to physical coordinates, yields stability of $u_T$. Finally, by means of Proposition \ref{Prop:Equiv_1forms}, from this we obtain Theorem \ref{Thm:Main}.

\subsection{Notation and conventions}

$\mathbb{B}^{d}_R$ denotes the $d$-dimensional ball in $\R^d$, centered at the origin, with radius $R$. For the unit ball we simply write $\mathbb{B}^{d}$. By $\Hb$ we denote the closed complex right half-plane. By $\mc S(\R^d)$ we denote the Schwartz space. 
For a domain $\Omega \in \R^d$ and a non-negative integer $k$ we denote by $\dot{H}^k(\Omega)$ (resp.~$H^k(\Omega)$) the standard homogeneous (resp.~inhomogeneous) Sobolev space with 
\begin{equation*}
	\| u \|^2_{\dot{H}^k(\Omega)} := \sum_{|\alpha|=k}\| \partial^\alpha u \|^2_{L^2(\Omega)}, \quad \text{and} \quad \| u \|^2_{H^k(\Omega)} := \sum_{n=0}^{k}\| u \|^2_{\dot{H}^n(\Omega)}.
\end{equation*}
On a Banach space $X$, by $\mc B(X)$ we denote the set of bounded linear operators. For a closed linear operator $(L,\mc D(L))$ on $X$  we denote by $\rho( L)$ the resolvent set of $L$, and $\sigma( L):= \mathbb{C} \setminus \rho( L)$ stands for the spectrum of $ L$, while $\sigma_p( L)$ denotes the point spectrum. Also, when $\la \in \rho( L)$ we use the following convention for the resolvent operator, $ R_{ L}(\la):=(\la -  L)^{-1}$. By $\ker L$ and $\rg L$ we denote respectively the kernel and the range of $L$. If $ L \in \mc B(X)$ we denote by $r(L)$ the spectral radius of $L$. We also use the usual asymptotic notation $a \lesssim b$ to denote $a \leq Cb$ for some implicit constant $C>0$. Also, we write $a \simeq b$ if $a \lesssim b$ and $b \lesssim a.$ For the Wronskian of two functions $f,g \in C^1(I), I \subseteq \R$, we use the following convention $W(f,g):=fg'-f'g$.

\section{Similarity variables and well-posedness in lightcones}\label{Sec:Well-posedness}

\noindent Due to the finite speed of propagation, evolution of \eqref{Eq:YM_equiv}-\eqref{Eq:YM_equiv_init_cond} inside a (backward) lightcone represents an independent dynamical system, and since we study stability of blowup at the origin, our aim is to localize the evolution to backward lightcones of points $(T,0)$, for $T>0$.
This can in fact be efficiently done by the so-called \emph{similarity variables}.

\subsection{Similarity variables}\label{Sec:SimVar}

\noindent  We define
\begin{equation}\label{Def:SS_var}
	\tau=\tau(t):=\ln T - \ln(T-t), \quad\rho=\rho(t,r):=\frac{r}{T-t}.
\end{equation}
Also, we rescale the dependent variables in the following way
\begin{equation}\label{Def:scaled_u}
	\psi(\tau,\rho):=(T-t)^2u(t,r) \quad \text{and} \quad \varphi(\tau,\rho):=(T-t)^3\partial_tu(t,r),
\end{equation}
where (based on \eqref{Def:SS_var})
\begin{equation*}
	t=T(1-e^{-\tau}) \quad \text{and} \quad r=T\rho e^{-\tau}.
\end{equation*}  
Now, note that for $0 < S < T$ the coordinate transformation \eqref{Def:SS_var} maps the truncated lightcone
$$	\Gamma_{T,S}:=\{ (t,r): t\in [0,S], 0 \leq r \leq T-t \}
$$
onto the cylinder
$$C_{\tau(S)}:=\{ (\tau,\rho) : \tau \in [0,\tau(S)],  0 \leq \rho \leq 1\},$$
and consequently the evolution of $(u,\partial_tu)$ inside $\Gamma_{T,S}$ in $t$ corresponds to the evolution of $(\psi,\varphi)$ inside $C_{\tau(S)}$ in $\tau$.
 We now derive the evolution equation for $(\psi,\varphi)$. First, note that Eq.~\eqref{Eq:YM_equiv} via \eqref{Def:SS_var} turns into
\begin{multline*}
	\big(\partial^2_\tau + 5 \partial_\tau + 2 \rho \partial_\tau\partial_\rho \big)\psi(\tau,\rho)=\\ \left((1-\rho^2)\partial^2_\rho + \tfrac{d+1}{\rho}\partial_\rho -6\rho \partial_\rho  - 6\right)\psi(\tau,\rho) + (d-2)\psi(\tau,\rho)^2\big(3-\rho^2 \psi(\tau,\rho)\big).
\end{multline*}
Furthermore, since $\varphi(\tau,\rho)=(\partial_\tau + \rho \partial_\rho + 2)\psi(\tau,\rho)$ 
we obtain the following evolution system
\begin{align}\label{Eq:MatrixEq}
\begin{pmatrix} 
\partial_{\tau}\psi \vspace{2mm} \\
 \partial_{\tau}\varphi
\end{pmatrix}
=
\begin{pmatrix}
\varphi-\rho\partial_\rho\psi - 2\psi \vspace{2mm}\\
\triangle^{d+2}_\rho \psi -\rho\partial_{\rho}\varphi - 3\varphi 
\end{pmatrix}
+
\begin{pmatrix} 0 \vspace{2mm}\\
(d-2)\psi^2(3-\rho^2\psi)
\end{pmatrix},
\end{align}
where we dropped the dependence on independent variables so as to avoid notational clutter, and
where $\triangle^{d+2}_\rho$ denotes the $(d+2)$-dimensional radial Laplace operator
\begin{equation*}
	\triangle^{d+2}_\rho:=\partial^2_{\rho}+\tfrac{d+1}{\rho}\partial_{\rho}.
\end{equation*}
Furthermore, there is the following initial condition
\begin{equation}\label{Eq:YM_init_cond_tau}
	\psi(0,\cdot)=T^2 u_0(T\cdot), \quad \varphi(0,\cdot)=T^3 u_1(T\cdot).
\end{equation}

 The next task is to define the notion of solution to the system \eqref{Eq:MatrixEq}. This then, via \eqref{Def:SS_var}-\eqref{Def:scaled_u}, yields a lightcone solution to Eq.~\eqref{Eq:YM_equiv}, and this in turn leads via \eqref{Def:Equiv_ansatz} to an equivariant lightcone solution of the system \eqref{Eq:YM_general}. In the following section we introduce the function spaces we use in the analysis of Eq.~\eqref{Eq:MatrixEq}.

\subsection{Functional set-up}\label{Sec:Funct_setup} Throughout the rest of the paper, unless we explicitly say otherwise, the dimension $d$ is assumed to be odd. We start with defining the radial Sobolev spaces on balls. Namely, for $k \in \mathbb{N}_0$ and $R>0$ we let
\begin{equation}\label{Def:Rad_Sobolev}
	H^k_r(\B^d_R):= \{ u:(0,R)\rightarrow \mathbb{C}: u(|\cdot|) \in H^k(\B^d_R) \}.
\end{equation}
This space has a Hilbert space structure which is inherited from $H^k(\B^d_R)$. Also, note that we identified a radial Sobolev function with its radial profile. For simplicity, we also introduce the following parameter
\begin{equation*}
	k_d:=\tfrac{d+1}{2}.
\end{equation*}
With these preparations we define the central space for our analysis
\begin{equation}\label{Def:H}
	\mathcal{H}:=H_r^{k_d} \times H_r^{k_d-1}(\mathbb{B}^{d+2}).
\end{equation}
This is a Hilbert space with the inner product
\begin{equation*}
	(\mb u | \mb v)_{\mc H} := \big(u_1(|\cdot|)\, |\, v_1(|\cdot|) \big)_{H^{k_d}(\mathbb{B}^{d+2})} + \big(u_2(|\cdot|)\, |\, v_2(|\cdot|) \big)_{H^{k_d-1}(\mathbb{B}^{d+2})},
\end{equation*}
and the corresponding norm
\begin{equation*}
	\| \mb u \|_\mc{H}^2 := (\mb u | \mb u)_{\mc H} = \| u_1(|\cdot|) \|^2_{H^{k_d}(\mathbb{B}^{d+2})} + \| u_2(|\cdot|) \|^2_{H^{k_d-1}(\mathbb{B}^{d+2})},
\end{equation*}
for $\mb u = (u_1,u_2)$ and $\mb v = (v_1,v_2)$. 
Also for later convenience we define
\begin{equation*}\label{Def:H_R}
	\mathcal{H}_R:=H_r^{k_d} \times H_r^{k_d-1}(\mathbb{B}^{d+2}_R),
\end{equation*}
with an analogous inner product and the norm it induces.
There are several more spaces necessary for our subsequent work. First, there is the space of smooth radial functions on balls
\begin{equation}\label{Def:C_r_inf}
	C_r^\infty(\overline{\B^d_R}):= \{ u \in C^\infty(\overline{\B^d_R}) : u\text{ is radial} \}.
\end{equation}
Also, we define the space of ``even" smooth functions
\begin{equation}\label{Def:C_e_inf}
	C_e^\infty [0,R]:= \{ u \in C^\infty [0,R] : u^{(2k+1)}(0)=0 \text{ for }  k \in \mathbb{N}_0 \}.
\end{equation}
In the sequel we need the following important correspondence between the spaces \eqref{Def:C_r_inf} and \eqref{Def:C_e_inf}.
\begin{lemma}\label{Lem:C_even}
	Let $u: \overline{\B^d_R} \rightarrow \mathbb{C}$ be radial with the radial profile $\hat{u}$, i.e., $u=\hat{u}(|\cdot|)$. Then $u \in C_r^\infty(\overline{\B^d_R})$ if and only if $\hat{u} \in C^\infty_e[0,R]$.
\end{lemma}
\begin{proof}
	Assume that $\hat{u} \in C^\infty_e[0,R]$. By repeated application of the fundamental theorem of calculus we infer the existence of $\tilde{u} \in C^\infty[0,R^2]$ for which $u=\tilde{u}(|\cdot|^2)$. Then smoothness of $u$ follows from smoothness of both $\tilde{u}$ and $x \mapsto |x|^2$.
	Now, for the reverse implication suppose $u \in C^{\infty}(\overline{\B^d_R})$. Define $v:[-R,R] \rightarrow \mathbb{C}$ by $v(s):=u(s,0,\dots,0)$. Then $v$ is even and smooth. From this it follows that $v^{(2k+1)}(0)=0$ for $k \in \mathbb{N}_0$, and since $\hat{u}=v|_{[0,R]}$ then $\hat{u} \in C^\infty_e[0,R]$.
\end{proof}
\noindent If $u \in H_r^k(\B^d_R)$ then by (radial) Sobolev extension of $u(|\cdot|)$ to $\R^d$ and convolution with a smooth radial mollifier we obtain a sequence of radial smooth functions which converges to $u(|\cdot|)$ in $H^k(\B^d_R)$. This, according to Lemma \ref{Lem:C_even}, implies that $C^\infty_e[0,R]$ is dense in $H^k_r(\B^d_R)$, and consequently, $C^\infty_e[0,1]^2$ is dense in $\mc H$. We will also need the space of ``odd" smooth functions
\begin{equation*}
	C_o^\infty [0,R]:= \{ u \in C^\infty [0,R] : u^{(2k)}(0)=0 \text{ for }  k \in \mathbb{N}_0 \}.
\end{equation*}

 \subsection{Well-posedness in lightcones}\label{Subsec:Well-posedness}
In this section we establish well-posedness of the system \eqref{Eq:MatrixEq} in $\mc H$. To achieve this, we use the semigroup theory, and we therefore proceed with writing the system \eqref{Eq:MatrixEq}-\eqref{Eq:YM_init_cond_tau} in an abstract ODE form. Namely, by letting
 \begin{equation}\label{Def:YM_vector}
 	\Psi(\tau):=
 	\begin{pmatrix}
 		\psi(\tau,\cdot)  \\
 		\varphi(\tau,\cdot)
 	\end{pmatrix} 
 \end{equation}
 we obtain the following evolution equation
 \begin{equation}\label{Eq:YM_system_abstract}
 	\begin{cases}
 		\Psi'(\tau)=\widetilde{\mb L}_0\Psi(\tau)+\mb N_0(\Psi(\tau))\\	
 		\Psi(0)=\mb U_0(T),
 	\end{cases}
 \end{equation}
 where for a 2-component function\footnote{When writing column vectors inline we put them in the row form.}  $\mb u(\rho)=(u_1(\rho),u_2(\rho))$ we have
 \begin{equation}
 	\widetilde{\bf{L}}_0\mb u(\rho):=
 	\begin{pmatrix}
 		u_2(\rho)-\rho u_1'(\rho)-2u_1(\rho) \vspace{1.5mm}\\
 		u_1''(\rho)+\tfrac{d+1}{\rho}u_1'(\rho)-\rho u_2'(\rho)-3u_2(\rho)
 	\end{pmatrix},\label{Def:L0}
 \end{equation}
 \begin{equation}\label{Def:N(u)}
 	{\bf{N}}_0(\mb u)(\rho):=
 	\begin{pmatrix}
 		0 \vspace{1mm}\\
 		(d-2)u_1(\rho)^2\big(3-\rho^2 u_1(\rho)\big)
 	\end{pmatrix},
 \end{equation}
 and the initial data is 
 \begin{equation}\label{Def:Initial_data_operator}
 	\mb U_0(T) := 
 	\begin{pmatrix}
 		T^2 u_0(T\cdot)\\
 		T^3u_1(T\cdot)
 	\end{pmatrix}.
 \end{equation}
 
 Following the standard approach, we first study the linear version of Eq.~\eqref{Eq:YM_system_abstract}, and prove the existence of linear flow in $\mc H$. More precisely, we show that the (closure of the) operator $\widetilde{\mb L}_0$ generates a strongly continuous, one-parameter semigroup of bounded operators on $\mc H$. For this, we first need to supply $\widetilde{\mb L}_0$ with a domain which is dense in $\mc H$. In fact, we simply let
\begin{equation*}
	\mathcal{D} (\widetilde{\bf{L}}_0):=  C^\infty_e[0,1]^2.
\end{equation*}
With these preparations at hand, we formulate the anticipated semigroup existence result.
\begin{proposition}\label{Prop:free_semigroup}
	The operator $\widetilde{\bf{L}}_0: \mathcal{D} (\widetilde{\bf{L}}_0) \subseteq \mathcal{H} \rightarrow \mathcal{H}$ is closable, and its closure ${\bf L}_0$  generates a strongly continuous one-parameter semigroup $(\textbf{\emph{S}}_0(\tau))_{\tau \geq 0}$ of bounded operators on $\mathcal{H}$. Furthermore, there exists $M > 0$ such that
	\begin{equation}\label{Eq:Growth_S_0}
	\| \textbf{\emph{S}}_0(\tau) \mb{\emph{u}}\|_{\mathcal{H}} \leq M e^{-\frac{3}{2}\tau}\| \mb{\emph{u}} \|_{\mc H},
	\end{equation}
	for all $\tau \geq 0$ and all $\mb{\emph{u}} \in \mc H$.
\end{proposition}
\noindent So as not to unnecessarily distract the reader right now, we defer the proof of this result to the next section, and we proceed now with developing a well-posedness theory for the full nonlinear system~\eqref{Eq:YM_system_abstract}. With Proposition \ref{Prop:free_semigroup} at hand, we reformulate Eq.~\eqref{Eq:YM_system_abstract} in the integral form \`a la Duhamel
\begin{equation}\label{Eq:Duhamel0}
	\Psi(\tau)=\mb S_0(\tau)\mb U_0(T) + \int_{0}^{\tau}\mb S_0(\tau-s)\mb N_0(\Psi(s))ds,
\end{equation}
by means of which we define local (strong) solutions in lightcones.
\begin{definition}\label{Def:Strong_sol}
	 Let $T,S>0$ with $T>S$. We say that $u: \Gamma_{T,S} \rightarrow \R$ is a \emph{(strong) solution} to the Cauchy problem \eqref{Eq:YM_equiv}-\eqref{Eq:YM_equiv_init_cond} if the corresponding $\Psi : [0,\tau(S)] \rightarrow \mathcal{H}$  belongs to $C([0,\tau(S)],\mathcal{H})$ and satisfies Eq.~\eqref{Eq:Duhamel0} for all $\tau \in [0,\tau(S)]$.
\end{definition}  
 Since the linear map $(u_0,u_1) \mapsto   \mb U_0(T):\mc H_T \rightarrow \mc H$ is bounded, and the  nonlinearity $\mb N_0$ is locally Lipschitz continuous in $\mc H$ (see Sec.~\ref{Sec:Nonlin_est}), a standard fixed point argument applied to Eq.~\eqref{Eq:Duhamel0} yields the following local well-posedness result.
\begin{theorem} \label{Thm:Well-posedness}
Let $T,R>0$, and let
\[
\mc B_{T,R}:=\{ {\bf u} \in \mc H_T : \|{\bf u}\|_{\mc H_T} < R \}.
\]
Then there exists $\tilde{T}=\tilde{T}(T,R)>0$ such that for every $(u_0,u_1) \in \mc B_{T,R}$ there is a unique strong solution to \eqref{Eq:YM_equiv}-\eqref{Eq:YM_equiv_init_cond} on $\Gamma_{T,\tilde{T}}$. Furthermore, the corresponding data-to-solution map 
\[
 (u_0,u_1) \mapsto \Psi : \mc B_{T,R} \rightarrow C([0,\tau(\tilde{T})],\mathcal{H})
 \]
 is Lipschitz continuous.
\end{theorem}
\begin{proof}
	\noindent Fix $T,R>0$. For simplicity, we denote $\mb u_0:=(u_0,u_1)$. First, note that there is $C=C(T)>0$ such that for $\mb u_0 \in \mc H_T$ we have that
	\begin{equation}
		\| \mb U_0(T) \|_{\mc H} \leq C \| \mb u_0 \|_{\mc H_T}.
	\end{equation}
	Then, we define the Banach space
	\begin{equation*}
		X_{\mc T}:=\{ \Psi \in C([0,\mc T],\mc H) :   \| \Psi \|_{X_{\mc T}}:= \sup_{\tau \in [0,\mc T]} \| \Psi(\tau) \|_{\mc H} \}.
	\end{equation*}
	Furthermore, we let
	$
	B_{\mc T,\mc R}:= \{ \Psi \in X_{\mc T}: \| \Psi \|_{X_{\mc T}} \leq \mc R \},
	$
	and we (formally) define
	\begin{equation*}
		K_{\mb u_0}(\Psi)(\tau):= \mb S_0(\tau)\mb U_0(T)+ \int_{0}^{\tau}\mb S_0(\tau-s)\mb N_0(\Psi(s))ds.
	\end{equation*}
	Now we prove that there is a small enough $\mc T$ for which the operator $K_{\mb u_0}$ is well defined on $B_{\mc T,N}$ for $N:=2MCR$ (where $M$ is from Proposition \ref{Prop:free_semigroup}), and has a fixed point there. For this, we need the fact that the nonlinearity $\mb N_0$ is locally Lipschitz. This simply follows from (the proof of) Lemma \ref{Lem:Nonlin_est}, where this property is established for a more complicated nonlinear operator $\mb N$, see \eqref{Def:L'}. Consequently,  we conclude that there is $C_1=C_1(T,R)>0$ such that
	\begin{equation*}
		\| \mb N_0(\mb u) - \mb N_0(\mb v) \|_{\mc H} \leq C_1 \| \mb u - \mb v \|_{\mc H},
	\end{equation*}
	for all $\mb u,\mb v \in \mc H$ with $\|\mb u \|_{\mc H}, \| \mb v \|_{\mc H} \leq N $. Now, for $\mb u_0 \in \mc B_{T,R}$ and $\Psi \in B_{\mc T,N}$ we have the following estimate
	\begin{align*}
		\| K_{\mb u_0}(\Psi)(\tau) \|_{\mc H} &\leq M e^{-\frac{3}{2}\tau}\| \mb U_0(T) \|_\mc H + M\int_{0}^{\tau}e^{-\frac{3}{2}(\tau-s)}\| \mb N_0(\Psi(s))\|_{\mc H}ds\\
		&\leq MC\| \mb u_0 \|_{\mc H_T} + MC_1 \int_{0}^{\tau}\| \Psi(s)\|_{\mc H}ds\\
		&\leq \tfrac{N}{2} + MC_1\| \Phi \|_{X_{\mc T}}\int_{0}^{\tau}ds \\
		& \leq \tfrac{N}{2} + MC_1 N \mc T,
	\end{align*}
	for all $\tau \in [0,\mc T]$. Furthermore, given $\Psi, \Phi \in B_{\mc T,N}$, we have that
	\begin{align*}
		\| K_{\mb u_0}(\Psi)(\tau) - K_{\mb u_0}(\Phi)(\tau)\|_{\mc H} 
		& \leq M\int_{0}^{\tau}e^{-\frac{3}{2}(\tau-s)}\| \mb N_0(\Psi(s)) - \mb N_0(\Phi(s))\|_{\mc H}ds\\
		&\leq MC_1 \int_{0}^{\tau}\| \Psi(s) - \Phi(s)\|_{\mc H}ds \\
		& \leq MC_1 \mc T \| \Psi - \Phi \|_{X_{\mc T}},	
	\end{align*}
	for all $\tau \in [0,\mc T].$ Then by taking $\mc T:= (2MC_1)^{-1}$ we see that given any $\mb u_0 \in \mc B_{T,R}$, the operator $K_{\mb u_0}$ is contractive on $B_{\mc T,N}$. Therefore, by the Banach fixed point theorem, there is a unique $\Psi= \Psi(\mb u_0) \in B_{\mc T,N}$ for which $\Psi= K_{\mb u_0}(\Psi)$, i.e., $\Psi$ satisfies Eq.~\eqref{Eq:Duhamel0} for all $\tau \in [0,\mc T].$ Such $\Psi$ is in fact a unique solution inside the whole space $\mc X_{\mc T}$; this follows from the local Lipschitz continuity of $\mb N_0$ and Gronwall's inequality. The time $\tilde{T}$, existence of which is asserted by the lemma, is therefore given by $\tilde{T}:= \tau^{-1}(\mc T)=T(1-e^{-\mc T})$. 
	
	It remains to prove the continuity of the data-to-solution map $\mb u_0 \mapsto \Psi(\mb u_0) : \mc B_{T,R} \rightarrow \mc X_{\mc T}$. This simply follows from the fact that $K_{\mb u_0}$ is contractive and the following estimate
	\begin{equation*}
		\|	K_{\mb u_0}(\Psi)(\tau) - K_{\tilde{\mb u}_0}(\Psi)(\tau)\|_{\mc H} = \| \mb S_0(\tau)(\mb U_0(T) - \tilde{\mb U}_0(T)) \|_{\mc H} \leq MC\| \mb u_0 - \tilde{\mb u}_0 \|_{\mc H_T},
	\end{equation*}
	where $\tilde{\mb U}_0(T)$ is the initial data operator relative to $\tilde{\mb u}_0$, see \eqref{Def:Initial_data_operator}.
\end{proof}

 Now, if for a particular choice of initial data, the existence time $\tilde{T}$ can be chosen arbitrarily close to $T$, then we say that the corresponding strong solution exists on the (whole) lightcone 
$$\Gamma_T:=\{ (t,r): t\in [0,T), r \leq T-t \}.$$ 
Furthermore, in this case the solution exists  on all lightcones contained in $\Gamma_T$, i.e., on $\Gamma_S$ for $S<T$, and this observation brings about the following definition.
\begin{definition}\label{Def:Blwup_time}
	For initial data  $(u_0,u_1):[0,\infty) \rightarrow \R^2$ that belong to  $\mc H_R$ for every $R>0$, we define
	\begin{equation}\label{Def:Blowup_time}
		T_{u_0,u_1}:=\sup \, \{ T>0 : \text{there exists a strong solution to } \eqref{Eq:YM_equiv}\text{-}\eqref{Eq:YM_equiv_init_cond} \text{ on } \Gamma_T\}.
	\end{equation} 
	Note that the local well-posedness ensures that the set in \eqref{Def:Blowup_time} is non-empty, and therefore $T_{u_0,u_1} > 0$.	If $T_{u_0,u_1} < \infty$ then we say that the corresponding solution \emph{forms a singularity (or blows up) at the origin} in finite time and we furthermore call $T_{u_0,u_1}$ the \emph{blowup time at the origin}. 
\end{definition}

\begin{definition}
	By means of \eqref{Def:Equiv_ansatz}, \eqref{Def:Equiv_init_data}, Definitions \ref{Def:Strong_sol} and \ref{Def:Blwup_time}, and Proposition \ref{Prop:Equiv_1forms}, we also define the notions of \emph{strong equivariant lightcone solutions} and \emph{blowup time at the origin} for the full Yang-Mills system \eqref{Eq:YM_general}-\eqref{Eq:YM_gen_init_cond}.
\end{definition}

  Finally, we address the choice of space $\mc H$. First of all, we note that  in the Euclidean space setting, wave equations are typically studied in the product homogeneous Sobolev spaces $\dot{H}^k \times \dot{H}^{k-1}$. However, on bounded domains, such homogeneous quantities are only semi-norms, and we therefore work with $H^{k} \times H^{k-1}$. 
  Concerning the Sobolev order $k$, in this paper we consider (mostly for technical reasons) only integer values.  Furthermore, we note that for $k > k_d$, the analogous well-posedness theory can be developed. Of course, with appropriate adjustments, which primarily concern the proper choice of the inner product $(\cdot|\cdot)_D$, see \eqref{Def:Inner_product_D} below. Also, if $k < k_d$ then the nonlinearity $\mb N_0$ is not locally Lipschitz continuous, and therefore some sort of Strichartz type space-time integrability should also be imposed in order to ensure local well-posedness.

 To conclude the discussion, we point out that below scaling, i.e., for $k < \frac{d}{2}-1 = k_d - \frac{3}{2}$, we can not expect to have any meaningful well-posedness theory at all.  
 Namely, for the three lowest odd dimensions, Cazenave, Shatah and Tahvildar-Zadeh \cite{CazShaTah98} showed ill-posedness by exhibiting two different (weak) lightcone solutions that originate from the same initial data.

\subsection{Existence of $C_0-$semigroup - Proof of Proposition \ref{Prop:free_semigroup}}
Here we follow the approach of Donninger and Sch\"orkhuber \cite{DonSch17} for the analogous problem for the power nonlinearity wave equations. However, since our setting is somewhat different, we provide detailed proofs to all the claims involved.  We start with defining, for $d \geq 3$ odd, the following formal differential operators
\begin{equation*}
	\mc Du(\rho):= \tfrac{1}{\rho}u'(\rho), \quad D_du(\rho):= \mc D^{\frac{d-3}{2}}\big(\rho^{d-2}u(\rho)\big).  
\end{equation*}
The operator $D_d$ is fundamental for our analysis, and we proceed with proving an important bijection property it satisfies.
\begin{lemma}\label{Lem:D_d_K_d}
	The operator $D_d$ bijectively maps $C^\infty_{e}[0,1]$ onto $C^\infty_{o}[0,1]$. Furthermore, the inverse operator is given by 
	\begin{equation}\label{Def:K_d}
	K_du(\rho):=\rho^{2-d}\mc K^{\frac{d-3}{2}}u(\rho), \quad \text{where}
 \quad	\mc Ku(\rho) := \int_{0}^{\rho}s u(s)ds.
	\end{equation}
	Here we define the values of $D_du$ and $K_du$ (and their derivatives) at the endpoints in terms of limits.
\end{lemma}
\begin{proof}
	Let $u \in C^\infty_{e}[0,1]$ and $w \in C^\infty_{o}[0,1]$ be arbitrary. Since $(\cdot)^{d-2}u \in  C^\infty_{o}[0,1]$ and the operator $\mc D^{(d-3)/2}$ preserves both parity and smoothness up to zero, we have that $D_du \in C^\infty_{o}[0,1].$ Similarly $K_d w \in C^\infty_{e}[0,1]$ and simply from the definitions of the operators we have that $D_d K_d w = w$ and $K_d D_d u = u$.
\end{proof}
 Now, by means of $D_d$ we introduce a sesquilinear form on $C^\infty_{e}[0,1]$,
\begin{equation}\label{Def:Inner_product_D}
	(\mb u| \mb v)_D:=(D_{d+2}u_1 | D_{d+2}v_1)_{\dot{H}^1(0,1)} + (D_{d+2}u_2 | D_{d+2}v_2)_{L^2(0,1)},
\end{equation}
and the corresponding semi-norm $\| \mb u \|_D:=\sqrt{(\mb u| \mb u)_D}$. Then we have the following important equivalence result.
\begin{proposition}\label{Prop:Equivalence_D}
	We have that 
	\begin{equation}\label{Eq:Equiv_norm}
		\| {\bf u} \|_D \simeq \| {\bf u} \|_{\mc H}
	\end{equation}
	for all ${\bf u} \in C^\infty_{\text{e}}[0,1]^2$.
\end{proposition}
\noindent 
The proof is somewhat long and furthermore contains no insight essential to the rest of this section. We therefore provide it at the end, see Sec.~\ref{Sec:Equivalence}.
 Proposition \ref{Prop:Equivalence_D} says that the semi-norm $\| \cdot \|_D$ is in fact a norm. Furthermore, the closure of $C^\infty_e[0,1]^2$ under $\| \cdot \|_D$ yields a Hilbert space which we, due to~\eqref{Eq:Equiv_norm}, identify with $\mc H$. Subsequently, the equivalence \eqref{Eq:Equiv_norm} holds throughout the whole space $\mc H$.
 For the ensuing analysis we also need the following lemma.
\begin{lemma}\label{Lem:Comm_rel}
	Let $u \in C^{\infty}(0,1)$ and let $\Lambda u(\rho):=-\rho u'(\rho)$. Then 
	\begin{equation*}
		D_d \Lambda u = \Lambda D_d u + D_du \quad \text{and} \quad D_d \triangle_{d} u = (D_du)'',
 	\end{equation*} 
 	where $\triangle_d$ is the $d$-dimensional radial Laplace operator $\triangle_du(\rho) := \rho^{1-d}\big(\rho^{d-1}u'(\rho)\big)' $
\end{lemma}
\begin{proof}
	The proof follows by induction.
\end{proof}
 With these preparations at hand, we prove the following dissipation property of the operator $\widetilde{\mb L}_0$.
\begin{lemma}\label{Lem:Dissipation}
	We have that
	\begin{equation}\label{Eq:Dissipative_est}
		\Re (\widetilde{\bf L}_0 {\bf u} | {\bf u})_D \leq -\tfrac{3}{2}\| {\bf u} \|^2_D	\end{equation}
	for all ${\bf u} \in \mc D(\widetilde{\bf L}_0)$.
\end{lemma}
\begin{proof}
	Let $\mb u \in \mc D(\widetilde{\mb L}_0)$. Then by means of Lemma \ref{Lem:Comm_rel} we have that 
	\begin{equation}\label{Eq:D_dL_0}
		D_{d+2}{\widetilde{\bf L}}_0\mb u=\mb A_0D_{d+2}\mb u,
	\end{equation}
where
	\begin{equation}\label{Eq:Sys_A0}
		\mb A_0 \mb w(\rho):= 
		\begin{pmatrix}
		w_2(\rho)-\rho w_1'(\rho)- w_1(\rho) \vspace{1.5mm}\\
		w_1''(\rho) - \rho w_2'(\rho)-2w_2(\rho)
		\end{pmatrix}.
	\end{equation} 
	This then yields
	\begin{gather*}
		\Re \big(D_{d+2}({\widetilde{\bf L}}_0\mb u)_1|D_{d+2}u_1\big)_{\dot{H}^1(0,1)} 
		= 
		\Re (w_2'|w_1')_{L^2(0,1)} + \Re (\Lambda w_1|w_1)_{\dot{H}^1(0,1)} - \| w_1 \|^2_{\dot{H}^1(0,1)},\\
		\Re \big( D_{d+2}({\widetilde{\bf L}}_0\mb u)_2|  D_{d+2}u_2\big)_{L^2(0,1)} 
		= 
		\Re ( w_1''| w_2)_{L^2(0,1)} + \Re ( \Lambda w_2| w_2)_{L^2(0,1)} - 2\|  w_2 \|^2_{L^2(0,1)}.
	\end{gather*}
	Furthermore, by partial integration we get
	\begin{gather*}
   \Re (\Lambda w_1|w_1)_{\dot{H}^1(0,1)} = -\tfrac{1}{2}\|    w_1 \|^2_{\dot{H}^1(0,1)}- \tfrac{1}{2}|w_1'(1)|^2,\\
    \Re ( \Lambda w_2| w_2)_{L^2(0,1)}=
    \tfrac{1}{2}\| w_2 \|^2_{L^2(0,1)} - \tfrac{1}{2} | w_2(1) |^2, \\
	\Re ( w_1''| w_2)_{L^2(0,1)} = -\Re (w_1'|w_2')_{L^2(0,1)} + \Re \big( w_1'(1)\overline{ w_2(1)}\big).		
	\end{gather*}
	Finally, by adding up we obtain
	\begin{equation}
		\Re (\widetilde{\bf L}_0 {\bf u} | {\bf u})_D 
		= 
		-\tfrac{3}{2}\| {\bf u} \|^2_D -\tfrac{1}{2}|w_1'(1)- w_2(1)|^2
		\leq -\tfrac{3}{2}\| {\bf u} \|^2_D.
	\end{equation}
\end{proof}

 We also need the following density property of $\widetilde{\mb L}_0$.

\begin{lemma}\label{Lem:Range_dense}
	 The set $\rg  \widetilde{\bf L}_0$ is dense in $\mc H$.
\end{lemma}
\begin{proof}
	Let ${\bf f} \in C^\infty_e[0,1]^2$. We prove that the equation 
	\begin{equation}\label{Eq:Eigenv_Range}
		\widetilde{\bf L}_0\bf u = \bf f
	\end{equation}
	  is solvable in $\mc D(\widetilde{\bf L}_0) = C^\infty_e[0,1]^2$. 
	First, to get rid of the apparent singularity at $\rho=0$, we apply the operator $D_{d+2}$ to both sides of Eq.~\eqref{Eq:Eigenv_Range}. According to \eqref{Eq:D_dL_0}-\eqref{Eq:Sys_A0}, this yields
	\begin{gather}
			w_2(\rho)-\rho w_1'(\rho)- w_1(\rho)= D_{d+2}f_1(\rho), \label{Eq:w_2}\\
			w_1''(\rho) - \rho w_2'(\rho)-2w_2(\rho)= D_{d+2}f_2(\rho),\nonumber
	\end{gather}
	for $\mb w = D_{d+2}\mb u$. Furthermore, from this system we obtain an equation for $w_1$
\begin{equation}\label{Eq:Du_1}
	(1-\rho^2)w_1''(\rho) - 4\rho w_1'(\rho) - 2w_1(\rho) = g(\rho),
\end{equation}
where
$
	g(\rho)=  2D_{d+2}f_1(\rho) +\rho (D_{d+2}f_1)'(\rho) + D_{d+2}f_2(\rho).
$
Now, note that $w_{1,0}(\rho):= \rho (1-\rho^2)^{-1}$ and $w_{1,1}(\rho):= (1-\rho^2)^{-1}$ form a fundamental solution system to (the homogeneous version of)  Eq.~\eqref{Eq:Du_1}. Then, since the Wronskian is $W(w_{1,0},w_{1,1})(\rho)=-(1-\rho^2)^{-2}$, by the variation of parameters formula we find a particular solution
\begin{align}
	w_1(\rho)&=
	-\frac{\rho}{1-\rho^2}\int_{0}^{1}(1-s)g(s)ds + \frac{\rho}{1-\rho^2}\int_{0}^{\rho}g(s)ds - \frac{1}{1-\rho^2}\int_{0}^{\rho}sg(s)ds \label{Eq:w_at_0} \\
	&=-\frac{1}{1+\rho}\int_{0}^{1}sg(s)ds -\frac{\rho}{1-\rho^2}\int_{\rho}^{1}g(s)ds + \frac{1}{1-\rho^2}\int_{\rho}^{1}sg(s)ds \label{Eq:w_at_1}.
\end{align}
Since $g \in C_o^\infty[0,1]$, from~\eqref{Eq:w_at_0} we infer smoothness of $w_1$ at $\rho=0$, with $w_1^{(2k)}(0)=0$, $k \in \mathbb{N}_0$. Also, from~\eqref{Eq:w_at_1} we have smoothness of $w_1$ at $\rho=1$. Consequently $w_1 \in C^\infty_o[0,1]$. Then, from Eq.~\eqref{Eq:w_2} we obtain $w_2$ and conclude that $ K_{d+2} \mb w$ belongs to $\mc D(\widetilde{\bf L}_0)$. Finally, according to \eqref{Eq:D_dL_0} and Lemma \ref{Lem:D_d_K_d}, we conclude that $\mb u := K_{d+2} \mb w$ solves Eq.~\eqref{Eq:Eigenv_Range}.
\end{proof}

\begin{proof}[Proof of  Proposition \ref{Prop:free_semigroup}]
From Lemmas \ref{Lem:Range_dense} and \ref{Lem:Dissipation} we have that $\widetilde{\mb L}_0$ is closable, and its closure, denoted by $\mb L_0$, satisfies $\rg \mb L_0=\mc H$. Furthermore, because of its closedness, the operator $\mb L_0$ satisfies the dissipation property \eqref{Eq:Dissipative_est} on the whole of $\mc D(\mb L_0)$. Now, the Lumer-Philips theorem (for the precise statement we use here see \cite{RenRog04}, p.~407, Theorem 12.22) yields the existence of a strongly continuous one-parameter semigroup $(\mb S_0(\tau))_{\tau \geq 0}$ of bounded operators on $\mc H$, which is generated by $\mb L_0$, and for which
	\begin{equation*}
	\| \mb S_0(\tau) \mb{u}\|_{D} \leq e^{-\frac{3}{2}\tau}\| \mb{u} \|_{D}
	\end{equation*}
for $\tau \geq 0$ and $\mb u \in \mc H$. Finally, via equivalence \eqref{Eq:Equiv_norm}, this yields Proposition \ref{Prop:free_semigroup}.	
\end{proof}

\subsubsection{Proof of Proposition \ref{Prop:Equivalence_D}}\label{Sec:Equivalence}
\noindent We first prove several forms of Hardy's inequality we need.
\begin{lemma}\label{Lem:Hardy_on_[0,1]}
	Let $\alpha> -\frac{1}{2}$, $\beta > -1$ and $n \in \mathbb{N}$. Then
	\begin{align}
		&\| (\cdot)^{\alpha} u \|_{L^2(0,1)} 
		\lesssim 
		|u(1)|+\| (\cdot)^{\alpha+1} u' \|_{L^2(0,1)} \label{Eq:Hardy(1)}
	\end{align}
	and
	\begin{align} 
		&|u(1)| 
		\lesssim
		\|(\cdot)^{\beta}u \|_{L^2(0,1)}
		+ \|(\cdot)^{\beta+1}u' \|_{L^2(0,1)} \label{Eq:Hardy(2)}
	\end{align}
	for all $u \in C^1[0,1]$. Furthermore,
	\begin{equation}\label{Eq:Hardy(3)}
		\| (\cdot)^{-n} u \|_{L^2(0,1)} \lesssim \| (\cdot)^{-n+1}u' \|_{L^2(0,1)}
	\end{equation}
	for all $u \in C^n[0,1]$ with $u^{(j)}(0)=0$, $j=0,1,\dots,n-1$.
\end{lemma}
\begin{proof}
	By partial integration, we have that
	\begin{align*}
		\int_0^{1} r^{2\alpha}|u(r)|^2 dr 
		&\lesssim 
		|u(1)|^2 + \int_{0}^{1} r^{2\alpha+1}|u(r)||u'(r)|dr \\
		&\lesssim 
		|u(1)|^2 + \varepsilon \int_0^{1} r^{2\alpha}|u(r)|^2 dr + \frac{1}{4\varepsilon} \int_0^{1} r^{2\alpha+2}|u'(r)|^2 dr.
	\end{align*}
	Then \eqref{Eq:Hardy(1)} follows by taking $\varepsilon$ small enough. For \eqref{Eq:Hardy(2)} we first observe that
	\begin{equation*}
		|u(1)| = \left|\int_{0}^{1}(r^{\beta+1}u(r))'dr \right|
		\lesssim \int_{0}^{1} r^{\beta}|u(r)| dr
		+ \int_{0}^{1} r^{\beta+1}|u'(r)|dr.
	\end{equation*}
	The claim then follows from the fact that the $L^1$ norm is controlled by the $L^2$ norm on bounded intervals. For the third estimate, by partial integration we have
	\begin{align*}
		\int_{0}^{1}\frac{|u(r)|^2}{r^{2n}}dr &= \lim_{a \rightarrow 0^+}\frac{-|u(r)|^2}{(2n-1)r^{2n-1}}\Bigg|^1_a + \frac{2}{2n-1}\int_{0}^{1} \frac{\Re \big(u(r)\overline{u'(r)}\big)}{r^{2n-1}}dr \\	
		&\leq  \frac{2}{2n-1}\left( \varepsilon\int_{0}^{1} \frac{ |u(r)|^2}{r^{2n}}dr + \frac{1}{4\varepsilon}\int_{0}^{1} \frac{ |u'(r)|^2}{r^{2n-2}}dr\right),
	\end{align*}
	and by taking $\varepsilon$ small enough \eqref{Eq:Hardy(3)} follows.
\end{proof}
We now prove two lemmas, which then together yield Proposition \ref{Prop:Equivalence_D}.
\begin{lemma}\label{Lem:Equiv(1)}
	Let $d \in \mathbb{N}$ and $k \in \mathbb{N}_0$. If $ k < \frac{d}{2}$ then
	\begin{equation*} 
		\| u(|\cdot|) \|_{H^k(\B^d)} \simeq \sum_{n=0}^{k} \|(\cdot)^{n+\frac{d-1}{2}-k}u^{(n)} \|_{L^2(0,1)},
	\end{equation*}
	for all $u \in C^{\infty}_e[0,1]$.
\end{lemma}
\begin{proof}
	We start with the observation that for $ n \in \mathbb{N}_0 $
	\begin{equation}\label{Eq:DerSquared}
		\sum_{|\alpha|=n} \big| \partial_{x}^{\alpha} u(|x|) \big|^2 \gtrsim \big|u^{(n)}(|x|) \big|^2
	\end{equation}
	uniformly in $x$.
	Indeed, we inductively get 
	\begin{equation}
		u^{(n)}(|x|) = \sum_{|\alpha|=n}\partial_x^\alpha u(|x|) \frac{n!}{\alpha!}\frac{x^{\alpha}}{|x|^n},
	\end{equation}
	and then from the fact that $\sum_{|\alpha|=n} \frac{n!}{\alpha!}\frac{x^{2\alpha}}{|x|^{2n}}=1$ and Cauchy Schwarz we have
	\begin{equation}
		\sum_{|\alpha|=n} \big|\partial_x^\alpha u(|x|)\big|^2 \gtrsim \left( \sum_{|\alpha|=n} \frac{n!}{\alpha!} \big|\partial_x^\alpha u(|x|) \big|^2 \right) \left( \sum_{|\alpha|=n} \frac{n!}{\alpha!}\frac{x^{2\alpha}}{|x|^{2n}} \right) \gtrsim \big|u^{(n)}(|x|)\big|^2.
	\end{equation}
	Now, by Lemma \ref{Lem:Hardy_on_[0,1]} and Eq.~\eqref{Eq:DerSquared} we have that for $ n \leq k$
	\begin{align*}
		\|(\cdot)^{n+\frac{d-1}{2}-k}u^{(n)} \|_{L^2(0,1)}
		&\lesssim
		\sum_{i=n}^{k-1}|u^{(i)}(1)| + \|(\cdot)^{\frac{d-1}{2}}u^{(k)} \|_{L^2(0,1)}
		\lesssim
		\sum_{i=n}^{k}\|(\cdot)^{\frac{d-1}{2}}u^{(i)} \|_{L^2(0,1)}	\\
		& \simeq
		\sum_{i=n}^{k} \| u^{(i)}(|\cdot|) \|_{L^2(\B^d)} 
		\lesssim
		\sum_{i=n}^{k}\|\sum_{|\alpha|=i}|\partial^\alpha u(|\cdot|)| \|_{L^2(\B^d)} \\
		&\lesssim
		\| u(|\cdot|) \|_{H^k(\B^d)}.
	\end{align*}
	To obtain the reverse estimate we use the fact that for $n \geq 1$
	\begin{equation}\label{Eq:DerSquared2}
		\sum_{|\alpha|=n} \big|\partial_x^\alpha u(|x|) \big|^2 \lesssim \sum_{j=1}^{n} |x|^{2(j-n)} \big|u^{(j)}(|x|) \big|^2.
	\end{equation}
	Indeed, we first inductively obtain the identity
	\begin{equation*}
		\sum_{|\alpha|=n}\partial_x^\alpha u(|x|) = \sum_{j=1}^{n} u^{(j)}(|x|)\frac{P_j(x)}{|x|^{2n-j}},
	\end{equation*}
	where $P_j$ are certain homogeneous polynomials of degree $n$. Then Eq.~\eqref{Eq:DerSquared2} follows from this and the fact that
	$
	|P_j(x)| \lesssim |x|^n.
	$
	Now, from Eq.~\eqref{Eq:DerSquared2} for $0 \leq n \leq k$ we have 
	\begin{align*}
		\| u(|\cdot|) \|_{\dot{H}^n(\B^{d})} &=\sum_{|\alpha|=n}\| \partial^{\alpha}u(|\cdot|) \|^2_{L^2(\B^d)}  
		\lesssim 
		\sum_{j=0}^{n}\| (\cdot)^{\frac{d-1}{2}+j-n}u^{(j)} \|^2_{L^2(0,1)}\\
		& \leq  \sum_{j=0}^{n}\| (\cdot)^{\frac{d-1}{2}+j-k}u^{(j)} \|^2_{L^2(0,1)},
	\end{align*}
	and by summing over $n$ we get the desired estimate.	
\end{proof}
\begin{lemma}\label{Lem:EquivD}
	For odd $d$ we have that
	\begin{equation}
		\| D_{d+2} u \|_{\dot{H}^1(0,1)} \simeq \sum_{n=0}^{\frac{d+1}{2}} \|(\cdot)^{n}u^{(n)} \|_{L^2(0,1)} 
	\end{equation}
	and
	\begin{equation}
		\| D_{d+2} u \|_{L^2(0,1)} \simeq \sum_{n=0}^{\frac{d-1}{2}} \|(\cdot)^{n+1}u^{(n)} \|_{L^2(0,1)}
	\end{equation}
	for all $u \in C^{\infty}_{e}[0,1]$.
\end{lemma}
\begin{proof}
	One direction is straightforward. Namely, the $``\lesssim "$ estimates directly follow from the observation that 
	\[
	D_{d+2}u(\rho)=\sum_{n=0}^{\frac{d-1}{2}}a_n \rho^{n+1}u^{(n)}(\rho) \quad \text{and} \quad  (D_{d+2}u)'(\rho)=\sum_{n=0}^{\frac{d+1}{2}}b_n \rho^{n}u^{(n)}(\rho)
	\]
	for some positive $a_n,b_n$. For the reverse estimates, we do the following. First, let $w:=D_{d+2}u$. Then we have that 
	\begin{equation}\label{Eq:n->K}
		\rho^{n+1}u^{(n)}(\rho)=\rho^{n+1}(K_{d+2}w)^{(n)}(\rho) = \sum_{j=(d-1)/2-n}^{(d-1)/2}\alpha_{n,j}\rho^{-2j}\mc K^jw(\rho),
	\end{equation}
	for some real $\alpha_{n,j}$.
	Now, by repeated application of Lemma \ref{Lem:Hardy_on_[0,1]} (Eq.~\eqref{Eq:Hardy(3)} in particular) we get that for $n = 0,1,\dots, (d-1)/2$
	\begin{equation*}
		\| (\cdot)^{n+1}u^{(n)} \|_{L^2(0,1)} \lesssim \| w \|_{L^2(0,1)} = \| D_{d+2} u \|_{L^2(0,1)}.
	\end{equation*}
	In the same manner, based on Eq.~\eqref{Eq:n->K}  we obtain
	\begin{equation*}
		\| (\cdot)^{n}u^{(n)} \|_{L^2(0,1)} \lesssim \| \rho^{-1} w\|_{L^2(0,1)} \lesssim \| w' \|_{L^2(0,1)} = \|  D_{d+2} u \|_{\dot{H}^1(0,1)},
	\end{equation*}
	for $n = 0,1,\dots, (d-1)/2$. The case $n=k_d$ has to be treated separately. Namely, then we have
	\begin{equation*}
		\rho^{k_d}u^{(k_d)}(\rho)= \sum_{j=0}^{k_d-1}\beta_{j}\rho^{-2j-1}\mc K^jw(\rho) + w'(\rho),
	\end{equation*}
	for some real $\beta_{j}$. And again by repeated application of Lemma \ref{Lem:Hardy_on_[0,1]} we get
	\begin{equation*}
		\| (\cdot)^{k_d}u^{(k_d)} \|_{L^2(0,1)} \lesssim \| \rho^{-1} w\|_{L^2(0,1)}  + \| w' \|_{L^2(0,1)} \lesssim \| w' \|_{L^2(0,1)} = \|  D_{d+2} u \|_{\dot{H}^1(0,1)}.
	\end{equation*}
	This finishes the proof.
\end{proof}
Finally, Proposition \ref{Prop:Equivalence_D} follows from Lemma \ref{Lem:EquivD} and Lemma \ref{Lem:Equiv(1)} in which we put $d+2$ instead of $d$, and successively let $k=(d+1)/2$ and $k=(d-1)/2$.

\section{Stability analysis}\label{Sec:Stability_analysis}

\noindent Now we turn to studying dynamics near the static profile 
\begin{equation*}
	\Psi_{\text{st}}(\rho)
	:=
	\begin{pmatrix}
	\alpha(d)\big(\rho^2+\beta(d)\big)^{-1} \vspace{1mm}\\
	2\alpha(d)\beta(d)\big(\rho^2+\beta(d)\big)^{-2}
	\end{pmatrix}
	=
	\begin{pmatrix}
	(T-t)^2u_T(t,r) \vspace{2mm}\\
	(T-t)^3\partial_tu_T(t,r)
	\end{pmatrix};
\end{equation*}
here $t=T(1-e^{-\tau})$ and $r=T\rho e^{-\tau}$.
 Namely, we consider perturbations
\begin{equation}\label{Def:Pert_ansatz}
	\Psi(\tau)=\Psi_{\text{st}}+\Phi(\tau),
\end{equation}
where
$\Phi(\tau):=(\phi_1(\tau,\cdot),\phi_2(\tau,\cdot))$.
The well-posedness theory from Sec.~\ref{Subsec:Well-posedness} guaranties local existence of strong solutions for any $\Phi(0) \in \mc H$. However, for studying stability of $\Psi_{\text{st}}$ under small perturbations by $\Phi(0)$ (and showing asymptotic stability in particular), it is more convenient to represent solutions in terms of the semigroup associated with the evolution of $\Phi$. To this end, we substitute the ansatz \eqref{Def:Pert_ansatz} into Eq.~\eqref{Eq:YM_system_abstract} and thereby obtain the following initial value problem
\begin{equation}\label{Eq:YM_system_pert}
\begin{cases}
\Phi'(\tau)=\widetilde{\mb L}\Phi(\tau)+\mb N(\Phi(\tau))\\	
\Phi(0)=\mb U(\mb v,T).
\end{cases}
\end{equation}
Here $\widetilde{\mb L}:=\widetilde{\mb L}_0+\mb{L}'$ for
\begin{gather}
{\bf{L}'}\mb u:=
\begin{pmatrix}
0 \\
V_du_1
\end{pmatrix},\label{Def:L'}
\quad \text{and} \quad
{\bf{N}}(\mb u):= 
\begin{pmatrix}
0 \\
N_d(u_1)
\end{pmatrix}, 
 \end{gather}
where
\begin{equation}\label{Def:V}
	V_d(\rho):=\frac{3(d-2)\alpha(d)\Big(2\beta(d)-\big(\alpha(d)-2\big)\rho^2\Big)}{(\rho^2+\beta(d))^2},
\end{equation}
and
\begin{equation}\label{Def:Nonlin_N}
	N_d(u_1)(\rho):=(d-2)u_1(\rho)^2 \left( \frac{3(\alpha(d)-1)\rho^2-3\beta(d)}{\rho^2+\beta(d)}+\rho^2u_1(\rho) \right).
\end{equation}
Furthermore, the initial data is
\begin{equation*}
	\mb U (\mb v,T)= \mb U_0(T)- \Psi_{\text{st}}=
	\begin{pmatrix}
		T^2u_0(T\cdot)-u_1(0,\cdot) \\
		T^3 u_1(T\cdot) - \partial_0 u_1(0,\cdot)
	\end{pmatrix}
	=
	\begin{pmatrix}
		T^2u_0(T\cdot) - u_0 \\
		T^3 u_1(T\cdot) - u_1
	\end{pmatrix}
	+
	\mb v,
\end{equation*}
where, for convenience, we denoted
\begin{equation}\label{Def:InitCond_v}
	\mb v := 
	\begin{pmatrix}
		u_0 - u_1(0,\cdot) \\
		u_1 - \partial_0 u_1(0,\cdot)
	\end{pmatrix}.
\end{equation}
Note that the operator $\mb L' : \mc H \rightarrow \mc H$ is bounded (in fact it is compact), and therefore the operator $\widetilde{\mb L}$ is closable with its closure being 
\begin{equation*}
	\mb L := \mb L_0 + \mb L', \quad \text{with} \quad \mc D(\mb L)=\mc D(\mb L_0).
\end{equation*} 
Now, as a direct consequence of the bounded perturbation theorem (\cite{EngNag00}, p.~158, Theorem 1.3), we obtain the following proposition.

\begin{proposition}
	The operator ${\bf{L}}:\mathcal{D}({\bf{L}}) \subseteq \mathcal{H} \rightarrow \mathcal{H}$ generates a strongly continuous one-parameter semigroup $(\textbf{\emph{S}}(\tau))_{\tau \geq 0}$ of bounded operators on $\mathcal{H}$. Furthermore, we have that
	\begin{equation}\label{Eq:Growth_Est_S}
		\| \textbf{\emph{S}}(\tau) \mb{\emph{u}}\|_{\mathcal{H}} \leq M e^{(-\frac{3}{2}+ M\| \bf L' \|)\tau}\| \mb{\emph{u}} \|_{\mc H}
	\end{equation}
	for $\tau \geq 0$ and  $\mb{\emph{u}} \in \mc H$.
\end{proposition}
\noindent With this result at hand, we express the system ~\eqref{Eq:YM_system_pert} in the integral form 
\begin{equation}\label{Eq:Duhamel(1)}
	\Phi(\tau)=\mb S(\tau)\mb U(\mb v,T) + \int_{0}^{\tau}\mb S(\tau-s)\mb N(\Phi(s))ds.
\end{equation}
 The central question to stability analysis of $\Psi_{\text{st}}$ is the one of growth properties of $ (\mb S(\tau))_{\tau \geq 0}$. The estimates obtained by the bounded perturbation theorem are typically too crude to be useful. Namely, \eqref{Eq:Growth_Est_S} implies that the growth bound of $ (\mb S(\tau))_{\tau \geq 0}$ is at most $-\frac{3}{2}+ M\| \bf L' \|$, which is an overshoot. By Hadamard's formula, on the other hand, we know that the growth bound is given in terms of the spectral radius of $\mb S(\tau)$. However, analyzing the spectrum of semigroups is in general quite difficult.
 
  In some cases, where there are nice mapping relations between the spectra of a semigroup and its generator, one can shift the analysis to the generator. This is, in fact, the case with our problem. We demonstrate this by proving an abstract spectral mapping theorem for a large class of semigroups, which in particular covers $(\mb S(\tau))_{\tau \geq 0}$, see Appendix \ref{App:Semigroup}. Namely, roughly speaking, a compact perturbation $\mb L'$ adds to the right of the spectrum of $\mb L_0$ eigenvalues only, which are furthermore in correspondence with the point spectrum of $\mb S(\tau)$ that lies outside of $\sigma(\mb S_0(\tau))$. This ultimately reduces matters to the study of the unstable point spectrum of $\mb L$, i.e., the one belonging to $\Hb:=\{ \la \in \C : \Re \la \geq 0 \}$, see Theorem \ref{Thm:Semigroups} for the precise statement.

 This simplification is, however, in general not as hopeful as it might seem at the first sight. Namely, the operator $\mb L$ is highly non-self-adjoint, so traditional (self-adjoint) spectral techniques are essentially of no use here, and hence analyzing the spectrum of such operators is in general a notoriously difficult problem. Nevertheless, in our particular setting, where the potential $V_d$ is of the rational function form, we can employ the spectral techniques we have been devising lately, see \cite{CosDonGloHua16,CosDonGlo17,ChaDonGlo17,Glo18,GloSch21}, to prove that the unstable spectrum of $\mb L$ consists only of one eigenvalue $\la=1$. This instability, however, is not a genuine one, as it is an artifact of the underlying time translation symmetry; consequently $\Psi_{\text{st}}$ is spectrally stable. This result is established in the following two sections, and it represents the core of our stability argument.
\subsection{Spectral analysis of $\mb L$}

 Eq.~\eqref{Eq:Growth_S_0} implies that the growth bound of $(\mb S_0(\tau))_{\tau \geq 0}$ is at most $-\frac{3}{2}$, and this, according to Hadamard's formula (see \cite{EngNag00}, p.~251, Proposition 2.2), gives 
$$r(\mb S_0(\tau))=e^{-\frac{3}{2}\tau},$$
which in turn, based on the spectral inclusion law
\begin{equation*}
	\{ e^{\tau \la}:\la \in \sigma(\mb L_0) \} \subseteq \sigma(\mb S_0(\tau)) ,
\end{equation*}
implies that
\begin{equation}\label{Eq:Spec_L_0}
	\sigma(\mb L_0) \subseteq \{ \la \in \mathbb{C} : \Re \lambda \leq -\tfrac{3}{2} \}.
\end{equation}
Furthermore, the facts that $V \in C_e^\infty[0,1]$ and that $H_r^{k_d}(\B^{d+2})$ is compactly embedded in $H_r^{k_d-1}(\B^{d+2})$ imply that the operator $\mb L'$ is compact. Therefore, perturbing $\mb L_0$ by $\mb L'$ introduces only eigenvalues into the half-plane $\{ \la \in \mathbb{C} : \Re \lambda > -\frac{3}{2} \}.$
In particular, the unstable spectrum of $\mb L$ consists of finitely many eigenvalues with finite multiplicity, see Theorem \ref{Thm:Semigroups}, i). The central result of this section is that, apart from the symmetry-induced eigenvalue $\la=1$, there are no other unstable spectral points of $\mb L$.
\begin{proposition}\label{Prop:unstable_spectrum}
	For every odd $d \geq 5$ we have that 
	\begin{equation*}
		\sigma_p({\bf L}) \cap \{ \la \in \mathbb{C} : \Re \la \geq 0  \} = \{ 1 \}.
	\end{equation*}
\end{proposition}
\begin{proof}
	If   
$		(\la- \mb L)\mb u=0   $
	for some $\la \in \mathbb{C}$ and $\mb u = (u_1,u_2) \in \mc D(\mb L)$, then by a straightforward calculation we get that the first component $u_1$ satisfies the following ODE 
	\begin{equation}\label{Eq:Eigenv}
		(1-\rho^2)u_1''(\rho)+\left(\frac{d+1}{\rho}-2(\la+3)\rho \right)u_1'(\rho)-(\la+2)(\la+3)u_1(\rho)+V_d(\rho)u_1(\rho)=0
	\end{equation}
weakly on the interval $(0,1)$. Since, by assumption, $u_1 \in H^{k_d}_r(\B^{d+2})$, we have that $u_1 \in H^{k_d}(\delta,1)$ for any $\delta >0$, and by Sobolev embedding it follows that $u_1 \in C^{k_d-1}(0,1)$. Consequently, $u_1$ satisfies Eq.~\eqref{Eq:Eigenv} classically. Furthermore, due to the smoothness of coefficients of the ODE \eqref{Eq:Eigenv} we have that $u_1 \in C^\infty(0,1)$. Now, if $ \la \in \Hb$ then we claim that in fact $u_1 \in C^\infty[0,1]$. Since both $\rho=0$ and $\rho=1$ are regular singular points of Eq.~\eqref{Eq:Eigenv}, to prove the claim we use the Frobenius method. 
The Frobenius indices at $\rho=0$ are $s_1=0$ and $s_2=-d$. Therefore, there are two linearly independent solutions to Eq.~\eqref{Eq:Eigenv} on the interval $(0,1)$, with the power series expansions near $\rho=0$ of the following forms
  \begin{equation}\label{Def:Frobenius_at_0}
	u_{1,1}(\rho)=\sum_{n=0}^{\infty}a_n\rho^n \quad \text{and} \quad u_{1,2}(\rho)=C\log(\rho)u_{1,1}(\rho)+\rho^{-d}\sum_{n=0}^{\infty}b_n\rho^n,
\end{equation}
for $a_0=b_0=1$ and some $C \in \C$. Since $u_{1,1}$ is analytic at $\rho=0$ and $u_{1,2}$ does not belong to $H_r^{k_d}(\B^{d+2})$ (due to the strong singularity at the origin), we conclude that $u_1$ is a constant multiple of $u_{1,1}$ and therefore $u_1 \in C^\infty[0,1)$. Now we treat the other endpoint. Since the Frobenius indices at $\rho=1$ are $s_1=0$ and $s_2=\frac{d-3}{2}-\la$, we distinguish two cases. First, we assume that
\begin{equation}\label{Eq:finite_set_la}
	\la \in \{ 0,1,\dots,\tfrac{d-3}{2} \}.
\end{equation}  
Then the Frobenius method guarantees existence of two linearly independent solutions of Eq.~\eqref{Eq:Eigenv} on $(0,1)$, which have the following power series expansions centered at $\rho=1$
  \begin{equation*}
	u_{1,1}(\rho)=(1-\rho)^{\frac{d-3}{2}-\la} \sum_{n=0}^{\infty}a_n(1-\rho)^n \quad \text{and} \quad u_{1,2}(\rho)=C\log(1-\rho)u_{1,1}(\rho)+\sum_{n=0}^{\infty}b_n(1-\rho)^n,
\end{equation*}
for $a_0=b_0=1$ and some $C\in \R$. If $C=0$ then both $u_{1,1}$ and $u_{1,2}$ are analytic near $\rho=1$ and so must be $u_1$. If $C \neq 0$ then $u_{1,2} \notin H^{k_d}_r(\B^{d+2})$, and $u_1$ must therefore be a constant multiple of $u_{1,1}$, which is analytic at $\rho=1$. Hence $u_1 \in C^\infty[0,1]$.
Now, if \eqref{Eq:finite_set_la} does not hold, then the two solutions are
  \begin{equation*}
	u_{1,1}(\rho)=(1-\rho)^{\frac{d-3}{2}-\la} \sum_{n=0}^{\infty}a_n(1-\rho)^n \quad \text{and} \quad u_{1,2}(\rho)=\sum_{n=0}^{\infty}b_n(1-\rho)^n,
\end{equation*}
with $a_0=b_0=1$. Since $u_{1,2}$ is analytic at $\rho=1$ and $u_{1,1} \notin H^{k_d}_r(\B^{d+2})$, we infer that $u_1$ must be a constant multiple of $u_{1,2}$. Therefore $u_1 \in C^\infty[0,1]$.

For the rest of the proof we adopt the terminology from \cite{CosDonGloHua16}. Namely, if for some $\la \in \Hb $, Eq.~\eqref{Eq:Eigenv} admits a non-zero solution $u_1 \in C^\infty[0,1]$ then we call this solution an \emph{unstable mode} and the corresponding $\la$ an \emph{unstable eigenvalue} of Eq.~\eqref{Eq:Eigenv}. As a matter of fact there is an unstable mode
\begin{equation}\label{Def:g}
	g_1(\rho):=\frac{1}{(\rho^2+ \beta(d))^2},
\end{equation}
which corresponds to $\la=1$. Furthermore, by letting $g_2(\rho):=\rho g_1'(\rho)+3g_1(\rho)$ we get the function $\mb g :=(g_1,g_2) \in \mc D(\mb L)$ for which $	(1- \mb L)\mb g=0$, and hence $1 \in \sigma_p(\mb L) $. 

Also, by following the terminology from \cite{CosDonGloHua16}, we say that $u_T$ is \emph{mode stable} if $\la=1$ is the only unstable eigenvalue of Eq.~\eqref{Eq:Eigenv}. Therefore, to finish the proof, it is enough to show mode stability of $u_T$ for odd $d \geq 5$.
Also, note that the notion of mode stability does not depend on the existence of  well-posedness theory for the underlying Cauchy problem. Therefore, in the rest of the proof we assume no parity restriction on $d$ and we establish the following claim.

\smallskip

\emph{Claim:} The solution $u_T$ is mode stable in all dimensions $d \geq 5$.
\smallskip

\emph{Proof:} For $d=5$, this claim was proved in \cite{CosDonGloHua16} (see Theorem 1.2), and we extend it here to all higher dimensions. However, we emphasize that, as opposed to \cite{CosDonGloHua16}, the presence of the additional free parameter $d$ complicates matters non-trivially. For that reason, we follow the approach in \cite{CosDonGlo17}, where we proved an analogous claim in the setting of supercritical wave maps. What is more, we at the same time improve the method from \cite{CosDonGlo17}, we simplify it and we make it more transparent and robust.

The first step is to ``remove" the unstable eigenvalue $\la=1$ by an adaptation of the well-known procedure from supersymmetric quantum mechanics. For a detailed discussion of the procedure, and for a step-by-step adaptation see \cite{CosDonGlo17}, Sec.~3. First of all, by a change of variable 
$
u_1(\rho)=\rho^{-2}v(\rho)
$
we obtain the canonical form of Eq.~\eqref{Eq:Eigenv} (see Eq.~2.7 in \cite{CosDonGlo17})
\begin{equation}\label{Eq:Eigen_canon}
	(1-\rho^2)v''(\rho)+\left(\frac{d-3}{\rho}-2(\la+1)\rho \right)v'(\rho)-\la(\la+1)v(\rho)+\left(V_d(\rho)-\frac{2d-4}{\rho^2} \right)v(\rho)=0.
\end{equation}
Now, by following the procedure in \cite{CosDonGlo17}, Sec.~3.2, we derive the co-called supersymmetric problem
\begin{equation}\label{Eq:SUSY}
	(1-\rho^2)\tilde{v}''(\rho)+\left(\frac{d-3}{\rho}-2(\la+1)\rho \right)\tilde{v}'(\rho)-\la(\la+1)\tilde{v}(\rho)+\left(\tilde{V}_d(\rho)+2\right)\tilde{v}(\rho)=0,
\end{equation}
where
\begin{equation*}
	\tilde{V}_d(\rho):=-\frac{\big(4\beta(d)-d+5\big)\rho^4 - 2\beta(d)\big(2\beta(d)-d+3\big)\rho^2 +3\beta(d)^2(d-1)}{\rho^2(\rho^2+\beta(d))^2}.
\end{equation*}
If $v_\lambda$ is an unstable mode of Eq.~\eqref{Eq:Eigen_canon} that corresponds to eigenvalue $\la \neq 1$, then 
\begin{equation}\label{Eq:SUSY_map}
	\tilde{v}_\la(\rho):=\rho^{-m}(1-\rho^2)^{-\frac{\la-1-m}{2}}	\Big(\partial_\rho - h(\rho)\Big)\rho^{m}(1-\rho^2)^{\frac{\la+1-m}{2}}v_\la(\rho)
\end{equation}
is the corresponding unstable mode for Eq.~\eqref{Eq:SUSY}; here $m=\frac{d-3}{2}$ and $h(\rho)=\hat{v}_1'(\rho)/\hat{v}_1(\rho)$ for $\hat{v}_1(\rho)=\rho^{m+2}(1-\rho^2)^{1-\frac{m}{2}}g_1(\rho)$. For $\la=1$, however, $v_1(\rho):=\rho^2 g_1(\rho)$ is mapped into the zero function via map \eqref{Eq:SUSY_map}, and this explains the ``removal" of $\la=1$ mentioned above.

In what follows, we prove that Eq.~\eqref{Eq:SUSY} has no unstable eigenvalues. Our approach is the following. We consider the series expansion of the (normalized) analytic solution of Eq.~\eqref{Eq:SUSY} at $\rho=0$, and we prove that if $\la \in \Hb$ then this function can not be analytically continued beyond $\rho=1$. To simplify matters, we first reduce this equation to the one with four regular singular points. Namely, we make the following change of variables (which preserves the analyticity of solutions at $\rho=0$ and $\rho=1$)
\begin{equation*}
	\rho=\sqrt{x}, \quad \tilde{v}(\rho)=\frac{x^{\frac{3}{2}}}{x+\beta(d)}y(x),
\end{equation*}
and thereby arrive at an equation for $y$ in its canonical Heun form
\begin{multline}\label{Eq:Heun}
	y''(x)+\left(\frac{d+4}{2x}+\frac{2\la+5-d}{2(x-1)}-\frac{2}{x+\beta(d)}\right)y'(x)\\
	+\frac{\la(\la+3)x+(\la+6)(\la+1)\beta(d)-2d+8}{4x(x-1)\big(x+\beta(d)\big)}y(x)=0.
\end{multline} 
The normalized analytic solution at $x=0$ of Eq.~\eqref{Eq:Heun} (which is also called the local Heun function) is given by the following power series
\begin{equation}\label{Def:Heun_series}
	y_{d,\la}(x)= \sum_{n=0}^{\infty}a_n(d,\la)x^n,
\end{equation}
where $a_0(d,\la)=1$, $a_1(d,\la)=A_{-1}(d,\la)$, and the rest of the coefficients obey the recurrence relation
\begin{equation}\label{Eq:RecRel}
	a_{n+2}(d,\la)=A_n(d,\la)a_{n+1}(d,\la)+B_n(d,\la)a_{n}(d,\la),
\end{equation}
where 
\begin{gather*}
	A_n(d,\la):=\frac{(\la+2n+8)(\la+2n+3)\beta(d)-(2n+4)(2n+d-2)}{2\beta(d)(n+2)(2n+d+6)},\\
	B_n(d,\la):=\frac{(\la+2n+3)(\la+2n)}{2\beta(d)(n+2)(2n+d+6)}.	
\end{gather*}
 Note that now, $\la$ is an unstable eigenvalue of Eq.~\eqref{Eq:SUSY} precisely when the radius of convergence of the series \eqref{Def:Heun_series} is larger than one. Therefore, we analyze the asymptotics of the series coefficients. Namely, we study the limiting behavior of the quotient
\begin{equation*}
	r_n(d,\la):=\frac{a_{n+1}(d,\la)}{a_{n}(d,\la)}.
\end{equation*} 
Since $\lim_{n \rightarrow \infty}A_n(d,\la)=1-\beta(d)^{-1}$ and $\lim_{n \rightarrow \infty}B_n(d,\la)=\beta(d)^{-1}$, the so-called characteristic equation of Eq.~\eqref{Eq:RecRel} is 
\begin{equation}\label{Eq:CharEq}
	t^2-\left(1-\tfrac{1}{\beta(d)}\right)t - \tfrac{1}{\beta(d)}=0.
\end{equation}
Furthermore, $t_1=1$ and $t_2=-\beta(d)^{-1}$ are the solutions to Eq.~\eqref{Eq:CharEq} and (by Poincar\'e's theorem for difference equations with variable coefficients) we conclude that either $a_n(d,\la)=0$ eventually in $n$ or 
\begin{equation}\label{Eq:Lim1}
	\lim\limits_{n \rightarrow \infty}r_n(d,\la)=1
\end{equation}
or
\begin{equation}\label{Eq:Lim2}
	\lim\limits_{n \rightarrow \infty}r_n(d,\la)=-\tfrac{1}{\beta(d)}.
\end{equation}
Our intention is to prove that \eqref{Eq:Lim1} holds for all $d \geq 6$ and $\la \in \Hb$.
We found it convenient to separately treat the cases  $d \geq 7$ and $d=6$.
\smallskip

\noindent \emph{The case $d \geq 7$}.
Note that, given $\la \in \Hb$, $a_n(d,\la)$ does not equal zero eventually in $n$, since otherwise by backward substitution in Eq.~\eqref{Eq:RecRel} we get that $a_0(d,\la)=0$, which is in conflict with the assumption that $a_0(d,\la)=1$. Therefore, either \eqref{Eq:Lim1} or \eqref{Eq:Lim2} is true. In what follows we show that \eqref{Eq:Lim2} can not hold. To that end we derive
from Eq.~\eqref{Eq:RecRel} a recurrence relation for $r_n$. Namely, we have
\begin{equation}\label{Eq:RecRel2}
	r_{n+1}(d,\la)=A_n(d,\la)+\frac{B_n(d,\la)}{r_n(d,\la)},
\end{equation}
where 
$
	r_0(d,\la)=A_{-1}(d,\la).
$
The explicit expression of $r_n(d,\la)$ gets increasingly complicated as $n$ increases, and to circumvent the issue of estimating such complicated expressions we do the following. We first construct for the recurrence relation \eqref{Eq:RecRel2} a ``simple" approximate solution $\tilde{r}_n$, whose limiting value is 1. We then prove perturbatively that the actual solution $r_n$ is  close to it (globally in $d$ and $\la$), thereby excluding the   possibility of \eqref{Eq:Lim2}.
The approximate solution we use is the following
 \begin{equation}\label{Eq:Quasi-Sol}
	\tilde{r}_n(d,\lambda):= \frac{\lambda^2}{2(n+1)(2n+d+4)}+\frac{(4n+7)\lambda}{2(n+1)(2n+d+4)}+\frac{2n+4}{2n+d+4}.
\end{equation}
We chose this expression such as to emulate the behavior of $r_n(d,\la)$ for different values of the participating parameters. Namely, since $r_n(d,\la)$ is a ratio of two polynomials in $\la$ whose degrees differ by two, it behaves like a second degree polynomial for large values of $\la$. In fact, a simple induction in $n$ based on \eqref{Eq:RecRel2} yields
\[
r_n(d,\la)=\frac{\lambda^2}{2(n+1)(2n+d+4)}+\frac{(4n+7)\lambda}{2(n+1)(2n+d+4)} + O_{n,d}(1)\quad \text{as} \quad\la \rightarrow \infty.	
\]
This gives rise to the first two terms in \eqref{Eq:Quasi-Sol}, and to get the last term, we perform similar analysis after fixing $\la=0$.
 Now, to measure the ``distance" between $r_n$ and $\tilde{r}_n$, we define the relative difference
\begin{equation}\label{Def:Delta} 
	\delta_n(d,\la):=\frac{r_n(d,\la)}{\tilde{r}_n(d,\la)}-1.
\end{equation}
We intend to show that $\delta_n$ is small uniformly in $d$ and $\la$, and to that end we substitute \eqref{Def:Delta} into Eq.~\eqref{Eq:RecRel2} and derive the recurrence relation for $\delta_n$
\begin{equation}\label{Eq:DeltaRec} 
	\delta_{n+1}=\varepsilon_n-C_n\frac{\delta_n}{1+\delta_n},
\end{equation}
where
\begin{equation}\label{Eq:Eps_and_C} 
	\varepsilon_n=\frac{A_n\tilde{r}_n+B_n}{\tilde{r}_n\tilde{r}_{n+1}}-1 \quad \text{and} 
	\quad C_n=\frac{B_n}{\tilde{r}_n\tilde{r}_{n+1}}.
\end{equation}
Now, for every $ d \geq 7, \la\in\Hb$ and $n\geq 1$, the following estimates hold
 \begin{align}\label{Eq:Estimates}
	|\delta_1(d,\la)|\leq\frac{1}{2}, \quad |\varepsilon_n(d,\la)|\leq \frac{5}{12}-\frac{4}{2d+1}, \quad |C_n(d,\la)|\leq\frac{1}{12}+\frac{4}{2d+1}. 
\end{align}
We note that so far (in our previous works) all of the proofs of this kind of estimates  rely crucially on the fact that the (analogues of the) three quantities above, $\delta_1(d,\la)$, $\varepsilon_n(d,\la)$ and $C_n(d,\la)$, can be expressed as ratios of polynomials belonging to $\mathbb{Z}[n,d,\la]$. This is, however, not the case here, due to the presence of the parameter $\beta(d)$, see \eqref{Def:Alpha_Beta}. To deal with this issue we do the following. We let $d$ depend on a new parameter, namely we let
\begin{equation*}
	d=d(w):=\frac{2(w^2+8w+14)}{2w+5},
\end{equation*}
for $w>0$. Consequently, from \eqref{Def:Alpha_Beta} we get that
\begin{equation}\label{Def:Beta(w)}
	\beta(d(w))=\frac{2(w+2)\big((w+3)\sqrt{3}+2(w+2)\big)}{3(2w+5)}.
\end{equation}
Therefore, all three quantities $\delta_1(d(w),\la)$, $\varepsilon_n(d(w),\la)$ and $C_n(d(w),\la)$ can be expressed as ratios of polynomials belonging to $\mathbb{Z}[n,w,\la,\sqrt{3}]$. With this we can carry out an analogous proof to the one in the previous works. 
 We illustrate the process on the second estimate in \eqref{Eq:Estimates}; the other two are established analogously. 
We first note that it is enough to establish the estimate for $w \geq \frac{7}{5}$, as this covers the interval $d \geq 7$. Also, note that the inequalities in \eqref{Eq:Estimates} are to be proven for $n \geq 1$. Therefore, we consider $\varepsilon_{n+1}\big(d(w+\frac{7}{5}),\la\big)$ (with shifted parameters) and prove the estimate for $n,w \geq 0$. Due to \eqref{Def:Beta(w)} we have that
\begin{equation}\label{Eq:eps_w}
	\varepsilon_{n+1}\big(d(w+\tfrac{7}{5}),\la\big)=\frac{P_1(n,w,\la)\sqrt{3}+P_2(n,w,\la)}{P_3(n,w,\la)\sqrt{3}+P_4(n,w,\la)},
\end{equation}
for some $P_j(n,w,\la) \in \mathbb{Z}[n,w,\la]$. Note that it is enough to establish the desired estimate for \eqref{Eq:eps_w} on the imaginary line, along with proving that \eqref{Eq:eps_w} is analytic and polynomially bounded for $\la\in\Hb$, as we can then by the Phragm\'en-Lindel\"of principle extend the estimate from the imaginary line to the whole of $\Hb$. To that end we compute
$$
|P_j(n,w,it)\sqrt{3}+P_{j+1}(n,w,it)|^2=:Q_j(n,w,t^2)\sqrt{3}+Q_{j+1}(n,w,t^2), 
$$
where $j \in \{ 1,3\}$, $t$ real, $Q_1,Q_2 \in \mathbb{Z}[n,w,t^2]$, and $Q_3,Q_4 \in \mathbb{N}[n,w,t^2]$. Now, the desired estimate is equivalent to
\begin{equation}\label{Eq:w_est}
\frac{Q_1(n,w,t^2)\sqrt{3}+Q_{2}(n,w,t^2)}{Q_3(n,w,t^2)\sqrt{3}+Q_{4}(n,w,t^2)} \leq \left(\frac{5}{12}-\frac{4}{2d(w+\tfrac{7}{5})+1} \right)^2 
	= \left(\frac{5}{12}\frac{100w^2+650w+1039}{100w^2+1130w+2911}\right)^2.
\end{equation}
Then, by a straightforward calculation we see that
\begin{multline*}
	[5(100w^2+650w+1039)]^2\,[Q_3(n,w,t^2)\sqrt{3}+Q_{4}(n,w,t^2)]\\-[12(100w^2+1130w+2911)]^2\,[Q_1(n,w,t^2)\sqrt{3}+Q_{2}(n,w,t^2)]\\
	=R_1(n,w,t^2)\sqrt{3}+R_2(n,w,t^2),
\end{multline*}
where $R_1$ and $R_2$ have manifestly positive coefficients. This proves the inequality \eqref{Eq:w_est}, and thereby for $d \geq 7$ and $n \geq 1$ the estimate for $\varepsilon_n(d,\la)$ in \eqref{Eq:Estimates} holds on the imaginary line. Since $\la \mapsto \varepsilon_n(d,\la)$ is obviously polynomially bounded in $\Hb$, it remains to prove that it is analytic there. This simply follows from \eqref{Eq:Eps_and_C} and the fact that both zeros of $\la \mapsto \tilde{r}_n(d,\la)$  are negative.  

Now, having established estimates \eqref{Eq:Estimates}, we employ a simple inductive argument to conclude from \eqref{Eq:DeltaRec} that
\begin{equation}\label{Eq:FinalEst}
	|\delta_n(d,\la)|\leq\frac{1}{2},
\end{equation}
for all $n\geq1, \la\in\Hb$ and $d \geq 7$. Since $\lim_{n\rightarrow\infty}\tilde{r}_n(d,\la)=1$, Eqs.~\eqref{Eq:FinalEst} and \eqref{Def:Delta} rule out \eqref{Eq:Lim2} and we therefore finally conclude that \eqref{Eq:Lim1} holds. This proves the mode stability claim for $d \geq 7$. 
\smallskip

\noindent \emph{The case $d = 6$}. In this case it suffices to prove the following estimates
 \begin{align}\label{Eq:Estimates_d=6}
	|\delta_1(6,\la)|\leq\frac{1}{3}, \quad |\varepsilon_n(6,\la)|\leq \frac{1}{12}+\frac{1}{4(n+1)}, \quad |C_n(6,\la)|\leq\frac{1}{2}-\frac{1}{2(n+1)}, 
\end{align}
for $\la \in \Hb$ and $n \geq 1$. The proof for the $\varepsilon_n$ estimate (and therefore for the other two as well) is analogous to the one above. Namely, we have that $\beta(6)=\frac{4}{3}+\frac{2}{3}\sqrt{6}$ and therefore the quantity $\varepsilon_{n+1}(6,\la)$ can be written in the following way
\begin{equation*}
	\varepsilon_{n+1}(6,\la)=\frac{P_1(n,\la)\sqrt{6}+ P_2(n,\la)}{P_3(n,\la)\sqrt{6}+ P_4(n,\la)},
\end{equation*}
for some $P_j(n,\la) \in \mathbb{Z}[n,\la]$. The rest of the proof goes the same as above. 
\end{proof}

\begin{remark}\label{Rem:Quasi}
	Although we treat an analogous problem, the procedure we used in the proof to construct the approximation $\tilde{r}_n$ differs from the one in \cite{CosDonGlo17}. The new method is simpler and more transparent; in particular it altogether avoids polynomial approximation techniques in order to find coefficients of the powers of $\la$ in \eqref{Eq:Quasi-Sol}. Furthermore, by allowing the bounds in \eqref{Eq:Estimates_d=6} to depend on $n$ as well, we managed to decrease the starting point of induction all the way to $n=1$, by means of which we avoided having to use Hurwitz stability criteria.
\end{remark}

\subsection{Estimates of the linear flow} 
From now on, for the sake of simplicity, we drop the subscript in $\| \cdot \|_{\mc H}$ and we simply write $\| \cdot \|$ for the norm relative to $\mc H$. 
 For the following result we need the Riesz projection operator relative to the unstable eigenvalue $\la=1$. Namely, we define
\begin{equation*}
	\mb P:= \frac{1}{2\pi i}\int_{\gamma}\mb R_{\mb L}(\la) \, d\la,
\end{equation*}
where $\gamma$ is a positively oriented circle centered at $1$ with radius $r_\gamma < 1$. Let us also recall the function $\mb g$ from the proof of Proposition \ref{Prop:unstable_spectrum}, namely
\begin{equation*}
	\mb g(\rho) = 
	\begin{pmatrix}
		g_1(\rho) \\
		\rho g_1'(\rho) + 3g_1(\rho)
	\end{pmatrix},
\end{equation*}
where $g_1$ is given in \eqref{Def:g}.
\begin{proposition}\label{Prop:Semigr_perturbed}
	For every odd $d \geq 5$ the following statements hold.
	\begin{itemize}
		\item[i)]  $\rg {\bf P} = \langle {\bf g} \rangle $. \\
		\item[ii)]  $ {\bf S}(\tau){\bf P}{\bf u}=e^\tau {\bf P}{\bf u}$ for all ${\bf u} \in \mc H$.  \\
		\item[iii)] There are $\omega>0$ and $C \geq 1$ such that 
		$$
		\| {\bf S}(\tau)(1- {\bf P}){\bf u}  \| \leq C e^{-\omega\tau} \| (1- {\bf P}){\bf u}  \|
		$$
		 for all ${\bf u} \in \mc H$. 
	\end{itemize}
\end{proposition}
\begin{proof}
	To prove i) we first show that
	\begin{equation}\label{Eq:Ker}
		\ker\, (1-\mb L)= \langle \mb g \rangle.
	\end{equation}
	By similar reasoning to the one in the first part of the proof of Proposition \ref{Prop:unstable_spectrum} we conclude that for any solution to the equation $(1-\mb L)\mb u=0$, we have that the first component $u_1$
 	satisfies Eq.~\eqref{Eq:Eigenv} for $\la=1$. Therefore
 	\begin{equation*}
 		u_1= c_1 u_{1,1} + c_2 u_{1,2}
 	\end{equation*} 	
 	(see \eqref{Def:Frobenius_at_0}) for some $c_1,c_2 \in \mathbb{C}$. Since $u_{1,2} \notin H^{k_d}_r(\B^{d+2})$, then $c_2=0$ and furthermore $u_1=c_1 u_{1,1}=\tilde{c}_1g_1$, for some $\tilde{c}_1 \in \C$. This proves \eqref{Eq:Ker}.
 	
	Now, by the definition  of the projection $\mb P$ we have that $\mb P \mb g = \mb g$, and consequently $\langle \mb g \rangle \subseteq \rg \mb P$. To prove the reversed inclusion we first show that the dimension of $\rg \mb P$ is finite. Indeed, otherwise $\la=1$ would belong to the essential spectrum of $\mb L$ (see \cite{Kat95}, p.~239, Theorem 5.28), and therefore to the spectrum of its compact perturbation $\mb L_0 = \mb L - \mb L'$, but this is  in conflict with \eqref{Eq:Spec_L_0}. The space $\rg \mb P$ reduces the operator $\mb L$ and we have that $\sigma(\mb L|_{\rg \mb P})= \{ 1 \}$, see \cite{Kat95}, Sec.~III.6.4, in particular p.~178, Theorem 6.17. Therefore, since $\text{dim} \rg \mb P$ is finite, the operator $1-\mb L|_{\rg \mb P}$ is nilpotent, i.e., there is $k \in \mathbb{N}$ for which $(1-\mb L|_{\rg \mb P})^k=0$. We show that $k=1$. To that end, we assume the contrary, that $k \geq 2$. Then there exists a nontrivial $\mb u = (u_1,u_2) \in \mc D(\mb L)$ for which \begin{equation}\label{Eq:Spec_nonhom}
		(1-\mb L)\mb u=-\mb g.
	\end{equation} 
	Consequently,
	\begin{equation}\label{Eq:u_1_cond}
		u_1 \in C^\infty(0,1) \cap H^{k_d}_r(\B^{d+2}),
	\end{equation}
	 and
	\begin{equation}\label{Eq:Eigenv_nonhom}
		(1-\rho^2)u_1''(\rho)+\left(\frac{d+1}{\rho}-8\rho \right)u_1'(\rho)-\Big( 12-V_d(\rho) \Big)u_1(\rho)=G(\rho),
	\end{equation}
	for $\rho \in (0,1),$ where $G(\rho)=2\rho g_1'(\rho)  + 7 g_1(\rho).$ Since $g_1$ is a solution to the homogeneous version of Eq.~\eqref{Eq:Eigenv_nonhom}, by reduction of order we obtain another one
	\begin{equation}\label{Def:tilde_g_1}
		\tilde{g}_1(\rho):= g_1(\rho)\int_{\rho_1}^{\rho}\frac{(1-x^2)^{\frac{d-7}{2}}}{x^{d+1}}\frac{dx}{g_1(x)^2},
	\end{equation}
	for an arbitrary choice of $\rho_1 \in (0,1)$. Note that since $g_1$ is even we have that $\tilde{g}_1(\rho)=\rho^{-d}A(\rho)$
	for some function $A$ which is analytic at $\rho=0$. Now, by the variation of constants formula we have that
	\begin{equation}\label{Eq:u_1_nonhom}
		u_1(\rho)= c_1 g_1(\rho) + c_2 \tilde{g}_1(\rho) + \tilde{g}_1(\rho) \int_{0}^{\rho} \frac{g_1(x)G(x)x^{d+1}}{(1-x^2)^{\frac{d-5}{2}}}dx - g_1(\rho) \int_{0}^{\rho} \frac{\tilde{g}_1(x)G(x)x^{d+1}}{(1-x^2)^{\frac{d-5}{2}}}dx,
	\end{equation}
	for a choice of $c_1,c_2 \in \C$. Due to a strong singularity of $\tilde{g}_1$ at $\rho=0$ we conclude that $c_2=0$, and then we separately treat cases $d=5$ and $d \geq 7$.
	
	For $d=5$, we choose $\rho_1=\tfrac{1}{2}$ in \eqref{Def:tilde_g_1}. Then $\tilde{g}_1(\rho) \simeq \log(1-\rho)$ for $\rho$ near 1. Furthermore, from \eqref{Eq:u_1_nonhom} it follows that $u_1$ has the same behavior near $\rho=1$, unless the integral multiplying $\tilde{g}_1$ is zero for $\rho=1$. But this is not the case as the integrand is strictly positive on $(0,1)$. Consequently, $u_1 \notin H_r^{3}(\B^7)$, which contradicts \eqref{Eq:u_1_cond}.
	
	For $d \geq 7$, we choose $\rho_1=1$ in \eqref{Def:tilde_g_1}, and since $g_1$ is non-vanishing at $\rho=1$ we have that 
	\begin{equation*}
		\tilde{g}_1(\rho)= (1-\rho^2)^{\frac{d-5}{2}}B(\rho),
	\end{equation*}
for some function $B$ which is analytic at $\rho=1$. Consequently, the last term in Eq.~\eqref{Eq:u_1_nonhom} is analytic at both $\rho=0$ and $\rho=1$. It remains to understand regularity properties of the remaining term
\begin{equation}\label{Def:I_d}
	\mc I_d(\rho):= \tilde{g}_1(\rho)\int_{0}^{\rho} \frac{g_1(x)G(x)x^{d+1}}{(1-x^2)^{\frac{d-5}{2}}}dx.
\end{equation}
We show that $\mc I_d$ is not analytic at $\rho=1$; this then implies that
\begin{equation}\label{Eq:I_d_asy}
	\mc I_d(\rho) \simeq (1-\rho)^{\frac{d-5}{2}}\log(1-\rho) \quad \text{near} \quad \rho=1,
\end{equation}  
and consequently $u_1 \notin H^{k_d}_r(\B^{d+2})$. Note that non-analyticity of $\mc I_d$ at $\rho=1$ is equivalent to the Laurent expansion at $x=1$ of the integrand in \eqref{Def:I_d}  having a nonzero coefficient multiplying $(1-x)^{-1}$. Although for a particular $d$ this can be evidenced by a straightforward calculation, proving it uniformly for all $d$ is a difficult problem. Interestingly, we managed to do this by relating $\mc I_d$ to a solution of the supersymmetric problem \eqref{Eq:SUSY} for $\la=1$, for which the non-analyticity at $\rho=1$ is known from Proposition \ref{Prop:unstable_spectrum}. For simplicity let $d=2m+3$. Then by integration by parts we have
\begin{align}
	\mc I_d(\rho)&=\tilde{g}_1(\rho)\int_{0}^{\rho}\frac{x^{2m+4}}{(1-x^2)^{m-1}}\big(  2xg_1'(x)g_1(x)+7g_1(x)^2  \big)dx \nonumber \\
	&=B(\rho)(1-\rho^2)^{m-1} \int_{0}^{\rho}\frac{x^{2m-2}}{(1-x^2)^{m-1}}\frac{d}{dx}\big(  x^7g_1(x)^2  \big)dx \nonumber \\
	&=B(\rho) \left[ \rho^{2m+5}g_1(\rho)^2-2(m-1)(1-\rho^2)^{m-1}\int_{0}^{\rho}\frac{x^{2m+4}}{(1-x^2)^{m}}g_1(x)^2dx \right]. \label{Eq:I_d}
\end{align} 
On the other hand, by a direct calculation we see that the function
\begin{equation*}
	\tilde{v}(\rho):=\frac{(1-\rho^2)^{m-1}}{\rho^{2m+2}g_1(\rho)}\int_{0}^{\rho}\frac{x^{2m+4}}{(1-x^2)^m}g_1(x)^2\, dx
\end{equation*}
solves Eq.~\eqref{Eq:SUSY} for $\la=1$. Note that $\tilde{v}\in C^\infty[0,1)$, and since, according to the proof of Proposition \ref{Prop:unstable_spectrum}, Eq.~\eqref{Eq:SUSY} has no unstable modes, $\tilde{v}$ is not analytic at $\rho=1$. From this and Eq.~\eqref{Eq:I_d} it follows that $\mc I_d$ is also not analytic at $\rho=1$; thereby~\eqref{Eq:I_d_asy} follows, and therefore $u_1 \notin H^{k_d}_r(\B^{d+2})$. This contradicts \eqref{Eq:u_1_cond}, and we finally conclude that $k=1$ for every odd $d \geq 5$. From this and Eq.~\eqref{Eq:Ker}  we have that $\rg \mb P \subseteq \ker (1-\mb L) \subseteq \langle \mb g \rangle $, and this finishes the proof of i).

Claim ii) follows from i) and the correspondence between point spectra of a semigroup and its generator. Claim iii) is a consequence of Proposition \ref{Prop:unstable_spectrum} and Theorem \ref{Thm:Semigroups}.
\end{proof}

\subsection{Estimates of the nonlinear terms}\label{Sec:Nonlin_est} In this section we prove local Lipschitz continuity of the nonlinear operator $\mb N$, see \eqref{Def:L'}.
 Throughout the rest of the paper we use the following definition 
 \begin{equation}\label{Def:B_delta}
 	\mc B_{\delta}:=\{ \mb u \in \mc H : \| \mb u \| \leq \delta \}.
 \end{equation}
\begin{lemma}\label{Lem:Nonlin_est}
	Let $\delta >0$. We have that
	\begin{equation}\label{Eq:Nonlin_est}
		\| \bf N(\bf u) -  N( v) \| \lesssim (\| 
		 u \| + \| 
		 v \|) \|  u - v \|,
	\end{equation}
	for all $\bf u,v \in{\bf{\mc B}}_\delta.$
\end{lemma}
\begin{proof}
	We first prove that given $f \in C^\infty_e[0,1]$ we have that
	\begin{equation}\label{Eq:Nonlin_u_1u_2u_3}
		\| fu_1 u_2 u_3 \|_{H^{k_d-1}(\B^{d+2})} \lesssim \prod_{i=1}^{3}\| u_i \|_{H^{k_d}(\B^{d+2})}
	\end{equation}
	for all $u_i \in H^{k_d}$, $i=1,2,3$. Let $\alpha_i \in \mathbb{N}_0^d$, $i=0,1,2,3$, with $\sum_{i=0}^{3}|\alpha_i| \leq k_d-1.$ Then by H\"older's inequality and Sobolev embedding we have that
	\begin{align*}
		\| \partial^{\alpha_0}f \prod_{i=1}^{3}\partial^{\alpha_i}u_i  \|_{L^2(\B^{d+2})} &\lesssim \prod_{i=1}^{3} \|\partial^{\alpha_i} u_i \|_{L^{\frac{2(d+2)}{1+2|\alpha_i|}}(\B^{d+2})}
		\lesssim \prod_{i=0}^{3} \| \partial^{\alpha_i} u_i \|_{H^{k_d-|\alpha_i|}(\B^{d+2})} \\
		&\lesssim \prod_{i=1}^{3}\| u_i \|_{H^{k_d}(\B^{d+2})},
	\end{align*}
	for all $u_i \in H^{k_d}$, $i=1,2,3$. From this estimate Eq.~\eqref{Eq:Nonlin_u_1u_2u_3} follows. Now, based on \eqref{Def:L'}-\eqref{Def:Nonlin_N}, the estimate \eqref{Eq:Nonlin_est} follows from Eq.~\eqref{Eq:Nonlin_u_1u_2u_3} and elementary identities $a^2-b^2=(a-b)(a+b)$ and $a^3-b^3=(a-b)(a^2+ab+b^2)$.
\end{proof}

\subsection{Contraction scheme}
 Let us recall the integral equation that governs the flow near $\Psi_{\text{st}}$
\begin{equation}\label{Eq:Duhamel_copy}
	\Phi(\tau)=\mb S(\tau)\mb U(\mb v,T) + \int_{0}^{\tau}\mb S(\tau-s)\mb N(\Phi(s))ds.
\end{equation}
In this section we prove the small data global existence and decay of solutions to Eq.~\eqref{Eq:Duhamel_copy}.
%
%
 An impediment to decay is, of course, the existence of the symmetry-induced instability $\la=1$ on the linear level. To deal with this, we use a Lyapunov-Perron-type argument. Namely, we first suppress the instability by conveniently modifying the initial data, and then later on we show that there is a choice of $T$ for which the correction term vanishes, thereby showing that given any small enough $\mb v$, there is a choice of $T$ which yields a unique, globally existing, and exponentially decaying solution to Eq.~\eqref{Eq:Duhamel_copy}.

 Let us make some technical preparations. First, we introduce the Banach space
\begin{equation*}
	\mc X := \{ \Phi \in C([0,\infty),\mc H) : 
	\| \Phi  \|_{\mc X} := \sup_{\tau>0}e^{\omega\tau}\|\Phi(\tau)\| < \infty
	 \}, 
\end{equation*}
where $\omega$ is from Proposition \ref{Prop:Semigr_perturbed}. 
Now, for $\mb u \in \mc H$ and $\Phi \in \mc X$
we define the correction term
\begin{equation}\label{Def:Correction_term}
	\mb C(\mb u, \Phi):= \mb P \left( \mb u + \int_{0}^{\infty}e^{- s}\mb N(\Phi(s))ds \right),
\end{equation}
and the map
\begin{equation*}
	\mb K_{\mb u}( \Phi)(\tau) := \mb S(\tau)\big(\mb u - \mb C(\mb u, \Phi) \big) + \int_{0}^{\tau}\mb S(\tau-s)\mb N(\Phi(s))ds.
\end{equation*}
In the sequel, we rely on a fixed point argument to prove the existence of globally small solutions to Eq.~\eqref{Eq:Duhamel_copy}, and for that we also need the following definition
\begin{equation*}
	\mc X_\delta := \{ \Phi \in \mc X: \| \Phi \|_{\mc X} \leq \delta \}.
\end{equation*}
Also, recall definition \eqref{Def:B_delta}.
\begin{proposition}\label{Prop:K}
	For all  sufficiently small $\delta>0$ and all  sufficiently large $C > 0$ the following holds: if ${\bf u} \in \mc B_{\delta/C}$ then there exits a unique $\Phi = \Phi({\bf u}) \in \mc X_\delta$ for which
	\begin{equation}\label{Eq:K(u,phi)}
		\Phi = {\bf K}_{\bf u}( \Phi).
	\end{equation}
	In addition, the map ${\bf u} \mapsto \Phi({\bf u}): \mc B_{\delta/C} \rightarrow \mc X$ is  Lipschitz continuous. 
\end{proposition}
\begin{proof}
	First, note that $\mb K_{\mb u}$ can be written in the following way
	\begin{equation*}
		\mb K_{\mb u}( \Phi)(\tau)= \mb S(\tau)(1- \mb P)\mb u +  \int_{0}^{\tau}\mb S(\tau-s) (1- \mb P)\mb N(\Phi(s))ds - \int_{\tau}^{\infty} e^{\tau- s} \mb P \mb N(\Phi(s))ds.
	\end{equation*}
   Then, from Proposition \ref{Prop:Semigr_perturbed} and Lemma \ref{Lem:Nonlin_est} we get that if $\Phi(s) \in \mc B_\delta$ for all $s \geq 0$ then
   \begin{equation*}\label{Eq:K(u,phi)_2}
    \| \mb K_{\mb u}( \Phi)(\tau) \| \lesssim e^{-\omega \tau}\| \mb u \| + e^{-\omega \tau} \int_{0}^{\tau}e^{\omega s}\| \Phi(s) \|^2ds + e^{\tau}\int_{\tau}^{\infty} e^{-s} \|  \Phi(s)  \|^2 ds.
   \end{equation*}
Furthermore, from this estimate we have that if $ \mb u \in \mc B_{\delta/C}$ and $\Phi \in \mc X_{\delta}$ then
\begin{equation*}
	e^{\omega \tau}\| \mb K_{\mb u}( \Phi)(\tau) \| \lesssim \tfrac{\delta}{C} + \delta^2 + \delta^2 e^{-\omega\tau},
\end{equation*}
where the implied constant can be chosen uniformly for all $\delta$ in a bounded subset of $\R^+$.
Also, in a similar manner we get that
\begin{equation*}
	e^{\omega \tau}\| \mb K_{\mb u}( \Phi)(\tau) - \mb K_{\mb u}( \Psi)(\tau) \| \lesssim (\delta + \delta e^{-\omega \tau})\| \Phi - \Psi \|_{\mc X},
\end{equation*}
for all $\Phi,\Psi \in \mc X_\delta$.
Finally, from the last two displayed equations we see that for all small enough $\delta$ and for all large enough $C$ the operator $\mb K_{\mb u}$ is contractive on $\mc X_\delta$ (with the contraction constant say $\frac{1}{2}$) provided $\mb u \in \mc B_{\delta/C}$, and the existence and uniqueness of solutions to Eq.~\eqref{Eq:K(u,phi)} follows from the Banach fixed point theorem. To prove continuity of the map $\mb u \mapsto \Phi(\mb u) $ we use Proposition \ref{Prop:Semigr_perturbed} and the contractivity of $\mb K_{\mb u}$. Namely, we have the estimate
\begin{align*}
	\| \Phi(\mb u)(\tau) - \Phi(\mb v)(\tau) \| &=
	\| \mb K_{\mb u}( \Phi(\mb u))(\tau) - \mb K_{\mb v}( \Phi(\mb v))(\tau)\| \\
	 &\leq  	\| \mb K_{\mb u}( \Phi(\mb u))(\tau) - \mb K_{\mb v}( \Phi(\mb u))(\tau)\| + \| \mb K_{\mb v}( \Phi(\mb u))(\tau) - \mb K_{\mb v} (\Phi(\mb v))(\tau)\|  \\
	 &\leq \| \mb S(\tau)(1-\mb P)(\mb u-\mb v) \| + \tfrac{1}{2}\| \Phi(\mb u)(\tau) - \Phi(\mb v)(\tau) \| \\ 
	&\leq  Ce^{-\omega\tau} \| \mb u - \mb v \| + \tfrac{1}{2}\| \Phi(\mb u)(\tau) - \Phi(\mb v)(\tau) \|,
\end{align*}
from which the Lipschitz continuity follows.
\end{proof}
 In order to use this result to construct solutions to Eq.~\eqref{Eq:Duhamel_copy} we have to ensure smallness of the initial data operator $\mb U(\mb v, T)$. The following lemma lays out this and a few other properties of $\mb U(\mb v, T)$ we need later on.
\begin{lemma}\label{Lem:U(v,t)}
	For $\delta \in (0,1)$ and ${\bf v} \in \mc H_{1+\delta} $ the map
	\begin{equation*}
		T \mapsto {\bf U}({\bf v},T):[1-\delta,1+\delta] \rightarrow \mc H
	\end{equation*}
	is continuous. In addition, we have that
	\begin{equation}\label{Eq:U(v,T)_est}
		\| {\bf U}({\bf v},T) \| \lesssim  \| {\bf v} \|_{\mc H_{1+\delta}} + |T-1|
	\end{equation}
	for all ${\bf v} \in \mc H_{1+\delta} $ and $T \in [1-\delta,1+\delta]$, where the implicit constant in this estimate can be chosen uniformly for all $\delta$ in a set positively distanced from 1.

\end{lemma}
\begin{proof}
	Fix $\delta \in (0,1)$ and $\mb v \in \mc H_{1+\delta}$. Then for $T,S \in [1-\delta,1+\delta]$ we have that
	\begin{equation}\label{Eq:T-S}
		 [\mb U(\mb v,T)]_1 - [\mb U(\mb v,S)]_1 = (T^2-S^2)u_0(T\cdot) + S^2 \big( u_0(T\cdot) - u_0(S \cdot) \big). 
	\end{equation}
	Let $\varepsilon >0$. Then there exists $\tilde{u}_0 \in C^\infty_e[0,1+\delta]$ for which $\| u_0(|\cdot|) - \tilde{u}_0(|\cdot|) \|_{ H^{k_d}(\B_{1+\delta}^{d+2})} < \varepsilon$. Then, by writing
	\begin{equation*}
		u_0(T\cdot) - u_0(S \cdot) = \big ( u_0(T\cdot)-\tilde{u}_0(T\cdot) \big) + \big( \tilde{u}_0(T\cdot) - \tilde{u}_0(S\cdot) \big) + \big( \tilde{u}_0(S\cdot) - u_0(S\cdot) \big)
	\end{equation*}
	and using the fact that 
	$\lim_{S \rightarrow T} \| \tilde{u}_0(T|\cdot|) - \tilde{u}_0(S|\cdot|)  \|_{ H^{k_d}(\B^{d+2})}=0,
	$
	from Eq.~\eqref{Eq:T-S} we see that 
	\begin{equation*}
		\lim_{S \rightarrow T} \|[\mb U(\mb v,T)]_1(|\cdot|) - [\mb U(\mb v,S)]_1(|\cdot|) \|_{H^{k_d}(\B^{d+2})} \lesssim \varepsilon;
	\end{equation*}
	the continuity follows by letting $\varepsilon \rightarrow 0$. For the second part, we write  $\mb U(\mb v,T)$ in the following way
	\begin{equation}\label{Eq:U(v,T)}
		\mb U(\mb v,T)=
		\begin{pmatrix}
			T^2v_1(T\cdot)\\
			T^3 v_2(T\cdot)
		\end{pmatrix}+ 
		\begin{pmatrix}
			T^2u_1(0,T\cdot)-u_1(0,\cdot) \\
			T^3 \partial_0u_1(0,T\cdot) - \partial_0 u_1(0,\cdot)
		\end{pmatrix},
	\end{equation}
and from this the estimate \eqref{Eq:U(v,T)_est} follows.
\end{proof}
 With these preparations at hand, we are ready to formulate and prove the main result of this section.
\begin{theorem}\label{Thm:CoMain}
	There  exist $\delta,N>0$ such that the following holds: if
	\begin{equation}\label{Eq:v_smallness}
		{\bf v} \in \mc H_{1+\delta/N}, \quad \text{the components of {\bf v} are  real-valued}, \quad \text{and} \quad \| {\bf v} \|_{\mc H_{1+\delta/N}} \leq \tfrac{\delta}{N^2},
	\end{equation}
	 then there exist $T \in [1-\frac{\delta}{N}, 1+\frac{\delta}{N}]$ and $ \Phi \in \mc X_{\delta}$ such that Eq.~\eqref{Eq:Duhamel_copy} holds for all $\tau \geq 0.$ Furthermore, such $\Phi$ yields via \eqref{Def:Pert_ansatz} a solution to Eq.~\eqref{Eq:Duhamel0} in $C([0,\infty),\mc H)$.
\end{theorem}
\begin{proof}
	For simplicity, set $\delta':=\delta/N$. Based on Lemma \ref{Lem:U(v,t)} and Proposition \ref{Prop:K} we infer that given small enough $\delta$ and large enough $N$, if $\mb v$ satisfies \eqref{Eq:v_smallness} and $T \in [1-\delta' , 1+\delta']$ then there is a unique $\Phi=\Phi({\mb v},T) \in \mc X_\delta $ that solves
	\begin{equation}\label{Eq:Duhamel_C}
		\Phi(\tau) =\mb S(\tau)\big(\mb U(\mb v,T) - \mb C(\mb U(\mb v,T), \Phi ) \big) + \int_{0}^{\tau}\mb S(\tau-s)\mb N\big(\Phi (s)\big)ds.
	\end{equation}
	Furthermore, since $\mb v$ is real-valued then so is $\Phi(\tau)$ for all $\tau \geq 0$. Then, to construct solutions to Eq.~\eqref{Eq:Duhamel_copy} we prove that there is a choice of $\delta$ and $N$ such that for any $\mb v$ that satisfies \eqref{Eq:v_smallness} there is $T=T(\mb v) \in[1-\delta', 1+\delta']$ for which
	the correction term in Eq.~\eqref{Eq:Duhamel_C} vanishes.
	%
	Since $\mb C$ takes values in $\rg \mb P = \langle \mb g \rangle$ it is enough to prove the existence of $T$ for which
		\begin{equation}\label{Eq:Fixed_pt}
	\big( \mb C(\mb U(\mb v,T), \Phi({\mb v},T))\, \big| \, \mb g \big)_{\mc H}=0.
	\end{equation} 
	To this end, we consider the real function $T \mapsto \big( \mb C(\mb U(\mb v,T), \Phi(\mb v,T))\, \big| \, \mb g \big)_{\mc H} $ and aim at using  Brouwer's fixed point theorem to prove that it vanishes on $[1-\delta', 1+\delta']$. The crucial observation is that
	\begin{equation*}
		\partial_T
		\begin{pmatrix}
			T^2u_1(0,T\cdot) \\
			T^3 \partial_0u_1(0,T\cdot)
		\end{pmatrix} \bigg|_{T=1} 
	=  \alpha  \mb g,
	\end{equation*}
	for some $\alpha >0$. Based on this, from \eqref{Eq:U(v,T)}, we get that 
	\begin{equation*}
		\big (\mb P\mb U(\mb v,T)\, \big| \, \mb g \big)_{\mc H}= \alpha(T-1)\| \mb g\|^2+R_1(\mb v,T),
	\end{equation*}
	where  $R_1(\mb v,T)$ is continuous in $T$ and $R_1(\mb v,T) \lesssim \delta/N^2$.
    Based on the definition of $\mb C$ we similarly conclude that
	\begin{equation*}
	\big( \mb C(\mb U(\mb v,T), \Phi(\mb v,T))\, \big| \, \mb g \big)_{\mc H}= \alpha (T-1)\| \mb g\|^2+ R_2(\mb v , T),
	\end{equation*}
 where $R_2(\mb v,T) \lesssim \delta/N^2 + \delta^2.$
	Subsequently, there is a sufficiently large $N$ and sufficiently small $\delta$ for which $|R_2(\mb v,T)| \leq \alpha \|\mb g \|^2 \delta'$. In that case Eq.~\eqref{Eq:Fixed_pt} becomes equivalent to
	\begin{equation}\label{Eq:T}
		T=F(T)
	\end{equation}
	for some continuous function $F$ which maps $[1-\delta',1+\delta']$ onto itself. Therefore, by Brouwer's fixed point theorem we infer the existence of $T \in [1-\delta',1+\delta']$ for which \eqref{Eq:T} (and therefore \eqref{Eq:Fixed_pt}) holds. Finally, after substituting the variation of constants formula
	\begin{equation*}
		\mb S(\tau)= \mb S_0(\tau) + \int_{0}^{\tau}\mb S_0(\tau-s)\mb L' \mb S(s)ds
	\end{equation*}
	in Eq.~\eqref{Eq:Duhamel_copy}, direct calculation yields that $\Psi:=\Psi_{\text{st}}+\Phi$ solves Eq.~\eqref{Eq:Duhamel0}.
\end{proof}
Finally, we prove the main theorem of the paper.
\begin{proof}[Proof of Theorem \ref{Thm:Main}]
	Due to equivalence \eqref{Eq:Equiv_1form}, we can choose $M$ large enough so that
	\begin{equation*}
		 \| A[0]-A_1[0] \|_{H^{k_d} \times H^{k_d-1}(\B^{d}_{1+\delta})} \leq \tfrac{\delta}{MN}
	\end{equation*} 
implies
	\begin{equation*}
		\| u[0]-u_1[0] \|_{H^{k_d} \times H^{k_d-1}(\B^{d+2}_{1+\delta})} \leq \tfrac{\delta}{N^2}  
	\end{equation*}
for $N$ and $\delta$ from Theorem \ref{Thm:CoMain}. We then let $\varepsilon:= \delta/N$. Now, for $\mb v=u[0]-u_1[0]$ (recall definition \eqref{Def:InitCond_v}) Theorem \ref{Thm:CoMain} guaranties the existence of $T \in [1-\varepsilon,1+\varepsilon]$ for which Eq.~\eqref{Eq:Duhamel0} has a solution $\Psi$ in $C([0,\infty),\mc H)$. This, according to Sec.~\ref{Subsec:Well-posedness}, implies the existence of a unique strong lightcone solution $A$ to \eqref{Eq:YM_general}-\eqref{Eq:YM_gen_init_cond} on $\Gamma_T$, which blows up at the origin at time $T$. What is more, for $i \in \{ 1,2,\dots,d\}$ and $k \in \{1,\dots, k_d\}$ we have that
	\begin{align*}
		(T-t)^{k+1-\frac{d}{2}} & \|(\cdot)_i \big(u(t,\cdot)-u_T(t,\cdot)\big) \|_{\dot{H}^k(\B^d_{T-t})} \\
		&= (T-t)^{k-1-\frac{d}{2}} \left\| (\cdot)_i \left(\psi \big(\tau(t),\tfrac{|\cdot|}{T-t} \big) - [\Psi_{\text{st}}]_1 \big(\tfrac{|\cdot|}{T-t} \big) \right) \right\|_{\dot{H}^k(\B^d_{T-t})} \\
		& = \| (\cdot)_i \left(\psi \big(\tau(t),|\cdot| \big) - [\Psi_{\text{st}}]_1(|\cdot|) \right) \|_{\dot{H}^k(\B^d)} = \| (\cdot)_i \phi_1 \big(\tau(t),|\cdot| \big) \|_{\dot{H}^k(\B^d)}  \\
		& \lesssim \| (\cdot)_i \phi_1 \big(\tau(t),|\cdot| \big) \|_{H^k(\B^d)}   \lesssim  \| \phi_1 \big(\tau(t),|\cdot| \big)\|_{H^k(\B^{d+2})}   \\
		& \lesssim  \| \phi_1 \big(\tau(t),|\cdot| \big)\|_{H^{k_d}(\B^{d+2})} \lesssim \| \Phi(\tau(t)) \| \lesssim  (T-t)^\omega,
	\end{align*}
where we used Proposition \ref{Prop:Equiv_unit_ball} in the next to the last line.
From this estimate and the fact that 
\begin{equation*}
\| A(t,\cdot)-A_T(t,\cdot) \|_{\dot{H}^k(\B^d_{T-t})} = (d-1) \sum_{i=1}^{d} \|(\cdot)_i \big(u(t,\cdot)-u_T(t,\cdot)\big) \|_{\dot{H}^k(\B^d_{T-t})},
\end{equation*}
we conclude that
\begin{equation*}
	(T-t)^{k+1-\frac{d}{2}}\| A(t,\cdot)-A_T(t,\cdot) \|_{\dot{H}^k(\B^d_{T-t})} \lesssim  (T-t)^\omega.
\end{equation*}
By the same  procedure, we get that
\begin{equation*}
	(T-t)^{k+1-\frac{d}{2}}\| \partial_tA(t,\cdot)- \partial_tA_T(t,\cdot) \|_{\dot{H}^{k-1}(\B^d_{T-t})} \lesssim  (T-t)^\omega.
\end{equation*}
 Finally, the last two displayed equations and the fact that $\| A_T[t] \|_{\dot{H}^k\times\dot{H}^{k-1}(\B^d_{T-t})} \simeq (T-t)^{\frac{d}{2}-1-k}$
 imply Eqs.~\eqref{Eq:MainThmEst} and \eqref{Eq:MainThmEst_fullnorm}.
\end{proof}

\appendix

\section{Sobolev norms of corotational maps on balls}\label{App:SobolevCorot}

\noindent Throughout this section, there is no restriction on the dimension $d$; it can be any positive integer. In what follows, we establish equivalence results for Sobolev norms of certain geometrically relevant maps, and, as a consequence, we obtain a proof of Proposition \ref{Prop:Equiv_1forms}. 

 Given $R>0$ we call a map $U: \B_R^d \rightarrow \C^d$ \emph{corotational} if there exists $u:[0,R) \rightarrow \C$  such that the coordinate functions of $U$ satisfy
\begin{equation}\label{Def:Corot_maps}
U_i(x):=x_iu(|x|), \quad x \in \B_R^d.
\end{equation}
Also, we call $u$ the \emph{radial profile of} $U$. Before we state a result about corotational maps, we prove an important auxiliary result.
\begin{lemma}\label{Lem:MainEquivBall}
	Let $k$ be a non-negative integer. If $k<\frac{d}{2}+1$ then
	\begin{equation*}
	\| u(|\cdot|) \|_{H^k(\B^d)} \simeq \sum_{n=0}^{k} \| (\cdot)^{\frac{d-1}{2}}u^{(n)} \|_{L^2(0,1)}
	\end{equation*}
	for all $u \in C^{\infty}_e[0,1]$.
\end{lemma}
\begin{proof}
	The ``$\gtrsim$" estimate follows from inequality \eqref{Eq:DerSquared}. The other direction is a consequence of the estimate \eqref{Eq:DerSquared2} and Lemma \ref{Lem:Hardy_on_[0,1]}. Indeed, for $k=0$ the assertion is obvious, while for $k \geq 1$ and $m \in \{ 1,\dots, k \}$ we have
	\begin{align*}
	\| u(|\cdot|) \|_{\dot{H}^m(\B^{d})} &= \sum_{|\alpha|=m}\| \partial^{\alpha}u(|\cdot|) \|^2_{L^2(\B^d)}  \lesssim \sum_{n=1}^{m}\| (\cdot)^{\frac{d-1}{2}+n-m}u^{(n)} \|^2_{L^2(0,1)} \\
	&\lesssim
	\sum_{n=1}^{m} \Big(  \sum_{i=n}^{m-1}|u^{(i)}(1)| + \|(\cdot)^{\frac{d-1}{2}}u^{(m)} \|_{L^2(0,1)} \Big)\\
	&\lesssim
	\sum_{n=1}^{m} \Big(  \sum_{i=n}^{m-1} \big(\| (\cdot)^{\frac{d-1}{2}}u^{(i)}\|_{L^2(0,1)} + \| (\cdot)^{\frac{d+1}{2}}u^{(i+1)}\|_{L^2(0,1)} \big) + \|(\cdot)^{\frac{d-1}{2}}u^{(m)} \|_{L^2(0,1)} \Big)\\
	& \lesssim
	\sum_{n=1}^{m}\|(\cdot)^{\frac{d-1}{2}}u^{(n)} \|_{L^2(0,1)},
	\end{align*}
	and by summing over $m$ we get the desired estimate.
\end{proof}
 The following proposition is the central result of this section; it relates Sobolev norms of corotational maps with the norms of their radial profiles.
\begin{proposition}\label{Prop:Equiv_unit_ball}
 Recall definition \eqref{Def:Corot_maps}. Let $k$ be a non-negative integer. If $k < \frac{d}{2}+2$ then
	\begin{equation*}
	\| U \|^2_{H^k(\B^d)} \simeq \| u(|\cdot|) \|^2_{H^k(\B^{d+2})}	
	\end{equation*}
	for all $u \in C^{\infty}_e[0,1].$
\end{proposition}
\begin{proof}
	Let $v \in C^\infty(\overline{\B^d})$. Since $\partial^\alpha_x(x_iv(x))=x_i \partial^{\alpha}v(x)+\alpha_i \partial^{\alpha-e_i}v(x)$, for a non-negative integer $n$ we have
	\begin{multline}\label{Eq:Corot(1)}
	\sum_{i=1}^{d} \sum_{|\alpha|=n} \frac{n!}{\alpha!} | \partial^{\alpha}_x(x_iv(x)) |^2 
	=  \sum_{|\alpha|=n} \frac{n!}{\alpha!} |x|^2 |\partial^{\alpha}v(x) |^2 \\
	+ \sum_{|\beta|=n-1}\frac{n!}{\beta!}\left((n-1+d)
	|\partial^{\beta}v(x)|^2 + x \cdot \nabla_x|\partial^\beta v(x)|^2 \right).
	\end{multline}
	Also, note that by divergence theorem we have that
	\begin{equation*}
		\| \partial^{\beta}v \|_{L^2(\mathbb{S}^{d-1})}^2 = \int_{\B^d}\nabla_x \cdot \big(x|\partial^{\beta}v(x)|^2\big) dx = \int_{\B^d} \big(d |\partial^{\beta}v(x)|^2 +  x \cdot \nabla_x|\partial^\beta v(x)|^2\big)dx.
	\end{equation*}
	Therefore, by integrating Eq.~\eqref{Eq:Corot(1)} over $\B^d$ we have
	\begin{multline*}
	\sum_{i=1}^{d} \sum_{|\alpha|=n} \frac{n!}{\alpha!} \| \partial^\alpha ((\cdot)_iv) \|_{L^2(\B^d)}^2 
	= \sum_{|\alpha|=n} \frac{n!}{\alpha!} \| |\cdot|\partial^\alpha v \|_{L^2(\B^d)}^2 \\
	+ \sum_{|\beta|=n-1}\frac{n!}{\beta!} \left((n-1)\| \partial^{\beta}v \|_{L^2(\B^d)}^2 + \| \partial^{\beta}v \|_{L^2(\mathbb{S}^{d-1})}^2 \right).
	\end{multline*}
	Now, by this, inequality \eqref{Eq:DerSquared}, and Lemma \ref{Lem:MainEquivBall} we get
	\begin{align*}
	\sum_{i=0}^{d} \| U_i \|^2_{H^k(\B^d)}
	& \gtrsim
	\sum_{n=0}^{k}\sum_{|\alpha|=n} \| |\cdot|\partial^\alpha [u(|\cdot|)] \|_{L^2(\B^d)}^2 
	= \sum_{n=0}^{k}\| |\cdot|^2 \sum_{|\alpha|=n}|\partial^\alpha [u(|\cdot|)]|^2 \|_{L^1(\B^d)} \\
	&\gtrsim  \sum_{n=0}^{k} \| |\cdot|^{2} |u^{(n)}(|\cdot|)|^2 \|_{L^1(\B^d)}
	\simeq	\sum_{n=0}^{k} \| (\cdot)^{\frac{d+1}{2}}u^{(n)} \|^2_{L^2(0,1)}
	\simeq \| u(|\cdot|) \|^2_{H^k(\B^{d+2})}.   		 
	\end{align*}
	For the other direction we prove several estimates. 
	First, for $|\alpha| \leq k $, from Eq.~\eqref{Eq:DerSquared2} and Lemmas \ref{Lem:Hardy_on_[0,1]} and \ref{Lem:MainEquivBall} we have that 
	\begin{equation}\label{Eq:Est1}
	\int_{\B^d}|x|^2 | \partial^\alpha_x u(|x|)|^2dx \lesssim \sum_{n=0}^{|\alpha|}\| (\cdot)^{\frac{d+1}{2}}u^{(n)} \|^2_{L^2(0,1)} \simeq \| u(|\cdot|) \|^2_{H^{|\alpha|}(\B^{d+2})}.
	\end{equation}
	Similarly for $|\beta|\leq k-1$ we have	\begin{align}\label{Eq:Est2}
	\int_{\B^d} x \cdot \nabla_x |\partial^\beta_x u(|x|)|^2 dx &\lesssim \int_{\B^d} |\partial^\beta_x u(|x|)|^2dx + \int_{\B^d}|x|^2 |\nabla_x (\partial^\beta_x u(|x|))|^2dx \nonumber \\ 
	&\lesssim \| u(|\cdot|) \|^2_{H^{|\beta|}(\B^d)} + \| u(|\cdot|) \|^2_{H^{|\beta|+1}(\B^{d+2})}.
	\end{align}
	Also, by Lemmas \ref{Lem:MainEquivBall} and \ref{Lem:Hardy_on_[0,1]} we have that for $n \leq k$
	\begin{align}\label{Eq:Est3}
	\|u(| \cdot |) \|_{H^{n-1}(\B^{d})} &\simeq 
	\sum_{i=0}^{n-1} \| (\cdot)^{\frac{d-1}{2}}u^{(i)} \|_{L^2(0,1)} 
	\lesssim \sum_{i=0}^{n-1} \Big( |u^{(i)}| + \| (\cdot)^{\frac{d+1}{2}}u^{(i+1)} \|_{L^2(0,1)} \Big) \nonumber \\
	&\lesssim \sum_{i=0}^{n-1} |u^{(i)}| + \sum_{i=1}^{n} \| (\cdot)^{\frac{d+1}{2}}u^{(i)} \|_{L^2(0,1)} 
	\lesssim \sum_{i=0}^{n} \| (\cdot)^{\frac{d+1}{2}}u^{(i)} \|_{L^2(0,1)} \nonumber \\
	& \simeq
	\| u(|\cdot|) \|_{H^n(\B^{d+2})}.
	\end{align}
	Finally, from the identity \eqref{Eq:Corot(1)} and Eqs.~\eqref{Eq:Est1}-\eqref{Eq:Est3} it follows that $\| U \|^2_{H^k(\B^d)} \lesssim \| u(|\cdot|) \|^2_{H^k(\B^{d+2})}.$
\end{proof}
By scaling the independent variable we obtain equivalence on balls of arbitrary radius.
\begin{corollary} \label{Cor:Equiv_Cor_R}
	Given non-negative integer $k < \frac{d}{2}+2$ and $R>0$, we have that
	\begin{equation}\label{Eq:Corot_Equiv}
	\| U \|_{H^k(\B^d_R)} \simeq \| u(|\cdot|) \|_{H^k(\B^{d+2}_R)}	
	\end{equation}
	for all $u \in C^{\infty}_e[0,R].$ Furthermore, the implicit constants in Eq.~\eqref{Eq:Corot_Equiv} can be chosen uniformly for all $R$ inside a bounded set which is positively distanced from zero.
\end{corollary}

\noindent
\begin{proof}[Proof of Proposition \ref{Prop:Equiv_1forms}]	
	The conclusion follows from Corollary \ref{Cor:Equiv_Cor_R} and the following observation $$ \| \emph{A} \|^2_{H^k(\B^d_R)}=	\sum_{i,j,k}\| \emph{A}^{ij}_k \|^2_{H^k(\B^d_R)} = (d-1) \sum_{i=1}^{d} \| (\cdot)_i u(|\cdot|) \|^2_{H^k(\B^d_R)} = (d-1)\| U \|^2_{H^k(\B^d_R)} .$$
\end{proof}
\begin{remark}
	Due to our crucial reliance on Hardy's inequality throughout this section, we had to impose restrictions on the size of $ k $. For example, extending the proof of Lemma \ref{Lem:MainEquivBall} to $k=\frac{d}{2}+1$ necessitates the estimate \eqref{Eq:Hardy(1)} for $\alpha=-\frac{1}{2}$, which, however,  does not hold. Nonetheless, we point out that there is no obvious reason not to believe that Proposition \ref{Prop:Equiv_unit_ball} and Corollary \ref{Cor:Equiv_Cor_R} hold for all $k \in \mathbb{N}_0$.
\end{remark} 
	 
As a curiosity, we finalize this section with a discussion on corotational maps defined on the whole $\R^d$. We note that by similar reasoning to the one above one can prove an analogous result to Proposition \ref{Prop:Equiv_unit_ball} for such maps. Interestingly, however, by exploiting the Fourier methods (to which we have access since we are treating the full space $\R^d$) we were able to prove the analogue of Proposition \ref{Prop:Equiv_unit_ball} without any size restriction on $k$.

\begin{proposition}\label{Prop:EquivSobNorm_full_space}
	Let $k \in \mathbb{N}_0$. Then
	\begin{equation*}
	\| U \|_{\dot{H}^k(\R^d)} \simeq \| u(|\cdot|) \|_{\dot{H}^k(\R^{d+2})}	
	\end{equation*}
	for all $u$ for which $u(|\cdot|) \in \mc S(\R^d)$.
\end{proposition}
\begin{proof}
	We use the Fourier transform based definition of homogeneous Sobolev norms and the Bessel function representation of the Fourier transform of radial functions. To be precise, we use the following definition of the Fourier transform in $d$ dimensions
	\[
	\mathcal{F}_df(\xi) := \int_{\R^d}f(x) e^{-ix\cdot \xi}dx.
	\]
	Then, for a radial Schwartz function $f=\tilde{f}(|\cdot|)$, we have that
	\begin{align}\label{Eq:RadialFourier}
	\mathcal{F}_df(\xi)=\frac{(2\pi)^\frac{d}{2}}{|\xi|^{\frac{d}{2}-1}}\int_{0}^{\infty}\tilde{f}(r)J_{\frac{d}{2}-1}(r|\xi|)r^{\frac{d}{2}}dr,
	\end{align}
	see e.g.~\cite{Gra08}, p.~429. Now, if $u$ is such that $u(|\cdot|) \in \mc S(\R^d)$, then we have that
	\begin{align*}
	\sum_{i=1}^{d}\big(\mathcal{F}_dU_i(\xi)\big)^2& = \sum_{i=1}^{d}\big(\mathcal{F}_d[(\cdot)_i u (|\cdot|)](\xi)\big)^2=\sum_{i=1}^{d}
\big(\partial_{\xi_i}\mathcal{F}_d[u(|\cdot|)](\xi)\big)^2\\
	&=
	|\nabla \mathcal{F}_d[u(|\cdot|)](\xi)|^2
	= \left(\frac{d}{d|\xi|}\mathcal{F}_d[u(|\cdot|)](\xi)\right)^2\\
	&= \left( \frac{(2\pi)^\frac{d}{2}}{|\xi|^{\frac{d}{2}-1}}\int_{0}^{\infty}u(r) \left(J'_{\frac{d}{2}-1}(r|\xi|)-\frac{\frac{d}{2}-1}{r|\xi|}J_{\frac{d}{2}-1}(r|\xi|) \right)r^{\frac{d}{2}+1} dr \right)^2\\
	&= \left( \frac{(2\pi)^\frac{d}{2}}{|\xi|^{\frac{d}{2}-1}}\int_{0}^{\infty}u(r)J_{\frac{d}{2}}(r|\xi|)r^{\frac{d+2}{2}} dr \right)^2 = \frac{|\xi|^2}{(2\pi)^2} \big(\mc F_{d+2}[u(|\cdot|)](\xi)\big)^2.
	\end{align*}
	Here we used the recurrence relation for $J$-Bessel functions
	\[
	J_{\nu+1}(z)=-J'_{\nu}(z)+\frac{\nu}{z}J_{\nu}(z).
	\]
	Finally,
	\begin{align*}
	\| U \|^2_{\dot{H}^k(\R^d)}
	=\sum_{i=1}^{d}\| U_i \|^2_{\dot{H}^k(\R^d)} 
	&\simeq \| |\cdot|^{2k}\sum_{i=1}^d(\mathcal{F}_dU_i)^2\|_{L^1(\R^d)}\\
	&\simeq
	\| |\cdot|^k \mathcal{F}_{d+2}[u(|\cdot|)] \|^2_{L^2(\R^{d+2})} 
	\simeq \| u(|\cdot|) \|^2_{\dot{H}^k(\R^{d+2})}.
	\end{align*}
\end{proof}
\begin{remark}
	Since the Fourier transform based definition of Sobolev spaces allows for fractional orders, the same result holds when $k$ is replaced by any real number greater than $-d/2$. 
\end{remark}

\section{On spectral mapping for compactly perturbed semigroups}\label{App:Semigroup}
\noindent For the statement of our result, we first define the Riesz projection operator. Namely, given an isolated spectral point $\la_0$ of a closed linear operator $\mb L$ on a Banach space $ X$, we define $\mb P_{\la_0} \in \mc B(X)$ by
\begin{equation*}
	\mb P_{\la_0}:= \frac{1}{2\pi i}\int_{\gamma}\mb R_{\mb L}(\la) \, d\la,
\end{equation*}
where $\gamma$ is a positively oriented circle centered at $\la_0$, contained in  the resolvent set of $\mb L$ and which, apart from $\la_0$, contains no other spectral point of $\mb L$ in its interior.
\begin{theorem}\label{Thm:Semigroups}
	Let $(X,\|\cdot\|)$ be a Hilbert space, and let ${\bf L}_0 : \mc D({\bf L}_0) \subseteq X \rightarrow X$ be a closed linear operator that generates a strongly continuous semigroup $ ({\bf S}_0(\tau))_{ \tau \geq 0} \subseteq \mc B(X)$ for which 
	\begin{equation}\label{Eq:S_0}
		\| {\bf S}_0(\tau) {\bf u}\| \leq M e^{\omega_0 \tau} \| {\bf u}\|,
	\end{equation}
	for some $M>0$, $\omega_0 \in \R$ and for all ${\bf u} \in X, \tau \geq 0$. Furthermore, let ${\bf L'}:X \rightarrow X$ be a compact operator. Then the operator 
	$$
	{\bf L}:={\bf L}_0 + {\bf L}': \mc D({\bf L}):= \mc D({\bf L}_0) \subseteq X \rightarrow X
	$$
	generates a strongly continuous semigroup $ ({\bf S}(\tau))_{ \tau \geq 0} \subseteq \mc B(X)$, for which, given any $\varepsilon>0$, the following statements hold.
	 \begin{itemize}
	 	\item [i)] The set
	 	\begin{equation*}
	 		S_\varepsilon := \sigma ({\bf L}) \cap \{ \la \in \C  : \Re \la \geq \omega_0 + \varepsilon   \}
	 	\end{equation*}
	 	consists of finitely many eigenvalues of ${\bf L}$, all of which have finite algebraic multiplicity. \\
	 	\item[ii)] The spectral mapping
	 	\begin{equation}\label{Eq:spec_mapping}
	 		 \sigma({\bf S}(\tau)) \setminus \mathbb{D}_{e^{\tau(\omega_0+\varepsilon)}} = \{ e^{\tau \la} :  \la \in S_\varepsilon  \}
	 	\end{equation}
 		holds for all $\tau > 0$, where $\mathbb{D}_{r}$ stands for the open disk of radius $r$ centered at zero.\\
 		\item[iii)] Let ${\bf P}:= \sum_{\la \in S_\varepsilon} {\bf P}_\la,$ and let $r_\varepsilon := \sup \, \{ \Re \la : \la \in \sigma({\bf L})\setminus S_\varepsilon  \} $.		
 		Then for every $\omega > \max \, \{ \omega_0, r_\varepsilon \}$ there exists $C \geq 1$ such that
 		\begin{equation*}
 			\| {\bf S}(\tau)(1-{\bf P}) {\bf u} \| \leq C e^{\omega \tau} \| (1-{\bf P}) {\bf u} \|,
 		\end{equation*}
 		for all ${\bf u} \in X$ and $\tau \geq 0$.
	 \end{itemize}
\end{theorem}

\begin{proof}
 	For the basic operator and semigroup theory facts that we use without proof, we will point to standard references \cite{EngNag00}, \cite{Kat95}, \cite{Sim15}. First, we note that 
	\begin{equation}\label{Eq:sigma(L_0)}
		\sigma ({\bf L}_0) \subseteq \{ \la \in \mathbb{C} : \Re \la \leq \omega_0 \}.
	\end{equation}
Indeed, from \eqref{Eq:S_0} we see that the growth bound of $(\mb S_0(\tau))_{\tau\geq 0}$ is at most $\omega_0$, and therefore by Hadamard's formula we have that
\begin{equation}\label{Eq:rad_S0}
	\frac{1}{\tau}\log r(\mb S_0(\tau)) \leq \omega_0, \quad \text{or equivalently} \quad r(\mb S_0(\tau)) \leq e^{\omega_0\tau},
\end{equation}	
for all $\tau >0$  (see \cite{EngNag00}, p.~251, Proposition 2.2); this then, according to the spectral inclusion relation
	\begin{equation*}
		\{ e^{\tau \la}: \la \in \sigma(\mb L_0) \} \subseteq \sigma(\mb S_0(\tau))
	\end{equation*}
	(see \cite{EngNag00}, p.~276, Theorem 3.6), implies \eqref{Eq:sigma(L_0)}.  
	The fact that $\mb L$ generates a semigroup follows simply from the bounded perturbation theorem (\cite{EngNag00}, p.~158, Theorem 1.3). 
  Now, denote $\mathbb{H}_{\omega_0}:=\{ \la \in \C : \Re \la > \omega_0 \}$. Then for $\la \in \mathbb{H}_{\omega_0}$ we have the following decomposition
  \begin{equation}\label{Eq:Resolvents}
  	\la - \mb L = (\la-\mb L_0) \big(1 -  \mb R_{\mb{L}_0}(\la)\mb L'\big).
  \end{equation}
 Also, from the compactness of $\mb L'$ we have that the map $\la \mapsto  \mb R_{\mb L_0}(\la)\mb L'$ is analytic on $\mathbb{H}_{\omega_0}$, and takes values in the set of compact operators on the Hilbert space $X$. Then by the analytic Fredholm theorem (see \cite{Sim15}, p.~194, Theorem 3.4.13) we infer the existence of a set $S \subseteq \mathbb{H}_{\omega_0}$,  which is discrete (in $\mathbb{H}_{\omega_0}$), and such that  $1 -  \mb R_{\mb{L}_0}(\la)\mb L'$ is invertible on $\mathbb{H}_{\omega_0} \setminus S$. Furthermore, $\la \mapsto \big(1-\mb R_{\mb L_0}(\la)\mb L'\big)^{-1}$ is analytic on $\mathbb{H}_{\omega_0} \setminus S$, with the points in $S$ being poles of finite order with finite rank residues. From this and \eqref{Eq:Resolvents} it follows that $\mb R_{\mb L}(\la)$ exists and is analytic on $\mathbb{H}_{\omega_0} \setminus S$, with $S$ consisting of finite-order poles with finite rank residues. This in turn implies that the elements of $S$ are isolated eigenvalues of $\mb L$ with finite (algebraic) multiplicities. We now prove ii) and then use that to finish the proof of i), i.e., to show that the set $S_\varepsilon \subseteq S$ is finite.
 To establish \eqref{Eq:spec_mapping}, we study the spectra of the semigroups. To start, from the variation of constants formula
	\begin{equation*}
		\mb S(\tau)= \mb S_0(\tau) + \int_{0}^{\tau}\mb S_0(\tau-s)\mb L' \mb S(s)ds
	\end{equation*}
	(see \cite{EngNag00}, p.~161, Corollary 1.7) and the compactness of $\mb L'$ we conclude that $\mb S(\tau)$ is a compact perturbation of $\mb S_0(\tau)$ for every $\tau \geq 0$ (see also \cite{EngNag00}, p.~258, Proposition 2.12). 
	This, by the same reasoning from above, implies  that $\sigma (\mb S(\tau)) \backslash \sigma(\mb S_0(\tau))$ consists of isolated eigenvalues, and from the spectral mapping theorem for the point spectrum (\cite{EngNag00}, p.~277, Theorem 3.7) the claim ii) follows. Furthermore, as the set \eqref{Eq:spec_mapping} is also closed (in $\C$) and bounded, it must be finite, for all $\tau > 0$. Note that this does not immediately imply the finiteness of $S_\varepsilon$, as this set might contain infinitely many points which are mapped into a single one via $\la \mapsto e^{\tau \la}$. This is, however, not possible because \eqref{Eq:spec_mapping} holds for all $\tau >0$. We demonstrate this in the following paragraph.
	
 We argue by contradiction. Namely, let us assume that the set $S_\varepsilon$ is infinite. Then, by finiteness of the set \eqref{Eq:spec_mapping} for $\tau=1$, we conclude
	that there is a point inside that set to which infinitely many points of $S_\varepsilon$ are mapped via $\la \mapsto e^\la$. This in turn implies the existence of $\alpha \in S_\varepsilon$ and an infinite set $A \subseteq \mathbb{Z}$ such that
	\begin{equation*}
		\{ \alpha + 2 \pi i k : k \in A \} \subseteq S_\varepsilon.
	\end{equation*}
	Now, by letting $\tau$ equal $\sqrt{2}$ (or any  positive irrational number) in \eqref{Eq:spec_mapping} we get that the infinite set $\{ e^{\sqrt{2}(\alpha + 2 \pi i k)} : k \in A \}$ is contained in the finite set $\sigma({\bf S}(\sqrt{2})) \setminus \mathbb{D}_{e^{\sqrt{2}(\omega_0+\varepsilon)}}$, a contradiction. This finishes the proof of i).
	
	 It remains to prove iii). First, we note that $\ker \mb P$ and $\rg \mb P$ are closed subspaces of $X$ and the following decomposition holds $X= \ker \mb P  \oplus  \rg \mb P$. Furthermore, the spaces $\ker \mb P$ and $\rg \mb P$ reduce $\mb L$ (i.e., they are invariant under the action of $\mb L$) and we accordingly define $\mb L_1$ and $\mb L_2$ as the restrictions of $\mb L$ on $\ker \mb P$ and $\rg \mb P$ respectively. More precisely, we let 
	$\mc D(\mb L_1):= \mc D(\mb L) \cap \ker \mb P$
	and define
	\begin{equation*}
		\mb L_1: \mc D(\mb L_1) \subseteq \ker \mb P \rightarrow \ker \mb P \quad \text{with} \quad \mb L_1 \mb u:=\mb L \mb u \text{~ for~ } \mb u \in \mc D(\mb L_1), 
	\end{equation*}
	with an analogous definition for $\mb L_2: \mc D(\mb L_2)\subseteq \rg \mb P \rightarrow \rg \mb P$.
	Furthermore, we have the following spectral decomposition
	\begin{equation}\label{Eq:spectrum_splitting}
		\sigma(\mb L_1)=\sigma(\mb L)\setminus S_\varepsilon \quad \text{and} \quad \sigma(\mb L_2)=S_\varepsilon.
	\end{equation}
 For proofs of these standard results, we refer the reader to \cite{Kat95}, Sec.~III.6.4, in particular p.~178, Theorem 6.17. 
 
  Now, the projection $\mb P$ commutes with $\mb L$ and its resolvent (see \cite{Kat95}, p.~173, Theorem 6.5) and therefore with $\mb S(\tau)$ for all $\tau \geq 0$. Consequently, $\ker \mb P$ and $\rg \mb P$ are $(\mb S(\tau))_{\tau\geq 0}$-invariant, closed subspaces of $X$, and therefore restrictions $(\mb S_1(\tau))_{\tau\geq 0}$ and $(\mb S_2(\tau))_{\tau\geq 0}$ on  $\ker \mb P$ and $\rg \mb P$ respectively are also strongly continuous one parameter semigroups of bounded operators on $\ker \mb P$ and $\rg \mb P$, which are furthermore generated by $\mb L_1$ and $\mb L_2$ respectively (see \cite{EngNag00}, p.~60, Sec.~2.3). 
 Furthermore, the restriction of the resolvent of $\mb S(\tau)$ on $\ker \mb P$ yields the resolvent of $\mb S_1(\tau)$, which implies that $\rho(\mb S(\tau)) \subseteq \rho(\mb S_1(\tau))$, or equivalently
 \begin{equation}\label{Eq:spectrum_inclusion}
 	\sigma(\mb S_1(\tau)) \subseteq \sigma(\mb S(\tau)),
 \end{equation} 
for all $\tau \geq 0$. Now, we claim that the growth bound of $(\mb S_1(\tau))_{\tau\geq 0}$ is at most $\max \, \{ \omega_0, r_\varepsilon \} $. Indeed, there would otherwise be a spectral point $\mu$ of $\mb S_1(\tau)$ with $|\mu| > e^{\tau \max \, \{ \omega_0, r_\varepsilon \}}$, which, based on \eqref{Eq:spectrum_inclusion} and \eqref{Eq:spec_mapping}, must be an eigenvalue of $\mb S(\tau)$ and therefore of $\mb S_1(\tau)$ as well, but by the spectral mapping theorem for the point spectrum this implies the existence of an eigenvalue of $\mb L_1$ with the real part greater than $\max \, \{ \omega_0, r_\varepsilon \}$, which is in contradiction with \eqref{Eq:spectrum_splitting}. Therefore, by the definition of growth bound, we conclude that for every $\omega > \max \, \{ \omega_0, r_\varepsilon \}$ there exists $C \geq 1$ such that
\begin{equation*}
	\| \mb S_1(\tau)\mb u \| \leq Ce^{\omega \tau} \|\mb u \|,
\end{equation*}
for all $\mb u \in \ker \mb P$. From this, the claim iii) follows. 
\end{proof}

\emph{Acknowledgment.} The author would like to thank Birgit Sch\"orkhuber and Matthias Ostermann for their useful comments on an early version of the paper.

\bibliography{refs-YM-higher-dim}

\begin{thebibliography}{10}

\bibitem{Act79}
Alfred Actor.
\newblock Classical solutions of $\mathrm{SU}(2)$ {Y}ang-{M}ills theories.
\newblock {\em Rev. Mod. Phys.}, 51:461--525, Jul 1979.

\bibitem{BieBiz15}
Pave{\l} Biernat and Piotr Bizo{\'n}.
\newblock Generic self-similar blowup for equivariant wave maps and
  {Y}ang-{M}ills fields in higher dimensions.
\newblock {\em Comm. Math. Phys.}, 338(3):1443--1450, 2015.

\bibitem{BieBizMal17}
Pawe{\l} Biernat, Piotr Bizo\'n, and Maciej Maliborski.
\newblock Threshold for blowup for equivariant wave maps in higher dimensions.
\newblock {\em Nonlinearity}, 30(4):1513--1522, 2017.

\bibitem{BizChm05}
P.~Bizo{\'n} and T.~Chmaj.
\newblock Convergence towards a self-similar solution for a nonlinear wave
  equation: A case study.
\newblock {\em Phys. Rev. D}, 72(4):045013, Aug 2005.

\bibitem{BizTab01}
P.~Bizo{\'n} and Z.~Tabor.
\newblock On blowup of {Y}ang-{M}ills fields.
\newblock {\em Phys. Rev. D (3)}, 64(12):121701, 4, 2001.

\bibitem{Biz00}
Piotr Bizo{\'n}.
\newblock Equivariant self-similar wave maps from {M}inkowski spacetime into
  3-sphere.
\newblock {\em Comm. Math. Phys.}, 215(1):45--56, 2000.

\bibitem{Biz02}
Piotr Bizo{\'n}.
\newblock Formation of singularities in {Y}ang-{M}ills equations.
\newblock {\em Acta Phys. Polon. B}, 33(7):1893--1922, 2002.

\bibitem{BizWas15}
Piotr Bizo\'{n} and Arthur Wasserman.
\newblock Nonexistence of shrinkers for the harmonic map flow in higher
  dimensions.
\newblock {\em Int. Math. Res. Not. IMRN}, (17):7757--7762, 2015.

\bibitem{CazShaTah98}
Thierry Cazenave, Jalal Shatah, and A.~Shadi Tahvildar-Zadeh.
\newblock Harmonic maps of the hyperbolic space and development of
  singularities in wave maps and {Y}ang-{M}ills fields.
\newblock {\em Ann. Inst. H. Poincar\'e Phys. Th\'eor.}, 68(3):315--349, 1998.

\bibitem{ChaDonGlo17}
Athanasios Chatzikaleas, Roland Donninger, and Irfan Glogi\'c.
\newblock On blowup of co-rotational wave maps in odd space dimensions.
\newblock {\em J. Differential Equations}, 263(8):5090--5119, 2017.

\bibitem{ChoChr81}
Yvonne Choquet-Bruhat and Demetrios Christodoulou.
\newblock Existence of global solutions of the {Y}ang-{M}ills, {H}iggs and
  spinor field equations in {$3+1$} dimensions.
\newblock {\em Ann. Sci. \'{E}cole Norm. Sup. (4)}, 14(4):481--506, 1981.

\bibitem{CosDonGlo17}
Ovidiu Costin, Roland Donninger, and Irfan Glogi{\'{c}}.
\newblock Mode stability of self-similar wave maps in higher dimensions.
\newblock {\em Comm. Math. Phys.}, 351(3):959--972, oct 2017.

\bibitem{CosDonGloHua16}
Ovidiu Costin, Roland Donninger, Irfan Glogi{\'c}, and M.~Huang.
\newblock On the {S}tability of {S}elf-{S}imilar {S}olutions to {N}onlinear
  {W}ave {E}quations.
\newblock {\em Comm. Math. Phys.}, 343(1):299--310, 2016.

\bibitem{CotKenMer08}
Rapha\"{e}l C\^{o}te, Carlos~E. Kenig, and Frank Merle.
\newblock Scattering below critical energy for the radial 4{D} {Y}ang-{M}ills
  equation and for the 2{D} corotational wave map system.
\newblock {\em Comm. Math. Phys.}, 284(1):203--225, 2008.

\bibitem{Don11}
Roland Donninger.
\newblock On stable self-similar blowup for equivariant wave maps.
\newblock {\em Comm. Pure Appl. Math.}, 64(8):1095--1147, 2011.

\bibitem{Don14a}
Roland Donninger.
\newblock Stable self-similar blowup in energy supercritical {Y}ang-{M}ills
  theory.
\newblock {\em Math. Z.}, 278(3-4):1005--1032, 2014.

\bibitem{DonSch17}
Roland Donninger and Birgit Sch{\"o}rkhuber.
\newblock Stable blowup for wave equations in odd space dimensions.
\newblock {\em Ann. Inst. H. Poincar\'e Anal. Non Lin\'eaire},
  34(5):1181--1213, 2017.

\bibitem{DonSch19}
Roland Donninger and Birgit Sch\"{o}rkhuber.
\newblock Stable blowup for the supercritical {Y}ang--{M}ills heat flow.
\newblock {\em J. Differential Geom.}, 113(1):55--94, 2019.

\bibitem{DonSchAic12}
Roland Donninger, Birgit Sch{\"o}rkhuber, and Peter~C. Aichelburg.
\newblock On stable self-similar blow up for equivariant wave maps: the
  linearized problem.
\newblock {\em Ann. Henri Poincar\'e}, 13(1):103--144, 2012.

\bibitem{Dum82}
Oana Dumitrascu.
\newblock Equivariant solutions of the {Y}ang-{M}ills equations.
\newblock {\em Studii \c si Cercet\u ari Matematice}, 34(4):329--333, 1982.

\bibitem{EarMon82a}
Douglas~M. Eardley and Vincent Moncrief.
\newblock The global existence of {Y}ang-{M}ills-{H}iggs fields in
  {$4$}-dimensional {M}inkowski space. {I}. {L}ocal existence and smoothness
  properties.
\newblock {\em Comm. Math. Phys.}, 83(2):171--191, 1982.

\bibitem{EarMon82b}
Douglas~M. Eardley and Vincent Moncrief.
\newblock The global existence of {Y}ang-{M}ills-{H}iggs fields in
  {$4$}-dimensional {M}inkowski space. {II}. {C}ompletion of proof.
\newblock {\em Comm. Math. Phys.}, 83(2):193--212, 1982.

\bibitem{EngNag00}
Klaus-Jochen Engel and Rainer Nagel.
\newblock {\em One-parameter semigroups for linear evolution equations}, volume
  194 of {\em Graduate Texts in Mathematics}.
\newblock Springer-Verlag, New York, 2000.
\newblock With contributions by S. Brendle, M. Campiti, T. Hahn, G. Metafune,
  G. Nickel, D. Pallara, C. Perazzoli, A. Rhandi, S. Romanelli and R.
  Schnaubelt.

\bibitem{Gas02}
Andreas Gastel.
\newblock Singularities of first kind in the harmonic map and {Y}ang-{M}ills
  heat flows.
\newblock {\em Math. Z.}, 242(1):47--62, 2002.

\bibitem{Glo18}
Irfan Glogi\'{c}.
\newblock {\em On the {E}xistence and {S}tability of {S}elf-{S}imilar {B}lowup
  in {N}onlinear {W}ave {E}quations}.
\newblock ProQuest LLC, Ann Arbor, MI, 2018.
\newblock Thesis (Ph.D.)--The Ohio State University.

\bibitem{GloSch20}
Irfan Glogi\'{c} and Birgit Sch\"{o}rkhuber.
\newblock Nonlinear stability of homothetically shrinking {Y}ang-{M}ills
  solitons in the equivariant case.
\newblock {\em Comm. Partial Differential Equations}, 45(8):887--912, 2020.

\bibitem{GloSch21}
Irfan Glogi\'{c} and Birgit Sch\"{o}rkhuber.
\newblock Co-dimension one stable blowup for the supercritical cubic wave
  equation.
\newblock {\em Adv. Math.}, 390:Paper No. 107930, 79, 2021.

\bibitem{Gra08}
Loukas Grafakos.
\newblock {\em Classical {F}ourier analysis}, volume 249 of {\em Graduate Texts
  in Mathematics}.
\newblock Springer, New York, second edition, 2008.

\bibitem{Gro01}
Joseph~F. Grotowski.
\newblock Finite time blow-up for the {Y}ang-{M}ills heat flow in higher
  dimensions.
\newblock {\em Math. Z.}, 237(2):321--333, 2001.

\bibitem{GunMar07}
Carsten Gundlach and Jos\'e~M. Mart\'in-Garc\'ia.
\newblock Critical phenomena in gravitational collapse.
\newblock {\em Living Reviews in Relativity}, 10(5), 2007.

\bibitem{Ham17}
Mark J.~D. Hamilton.
\newblock {\em Mathematical gauge theory}.
\newblock Universitext. Springer, Cham, 2017.
\newblock With applications to the standard model of particle physics.

\bibitem{Kat95}
Tosio Kato.
\newblock {\em Perturbation theory for linear operators}.
\newblock Classics in Mathematics. Springer-Verlag, Berlin, 1995.
\newblock Reprint of the 1980 edition.

\bibitem{Kla97}
S.~Klainerman.
\newblock On the regularity of classical field theories in {M}inkowski
  space-time {${\bf R}^{3+1}$}.
\newblock In {\em Nonlinear partial differential equations in geometry and
  physics ({K}noxville, {TN}, 1995)}, volume~29 of {\em Progr. Nonlinear
  Differential Equations Appl.}, pages 29--69. Birkh\"auser, Basel, 1997.

\bibitem{KlaMac95}
S.~Klainerman and M.~Machedon.
\newblock Finite energy solutions of the {Y}ang-{M}ills equations in {$\bold
  R^{3+1}$}.
\newblock {\em Ann. of Math. (2)}, 142(1):39--119, 1995.

\bibitem{KlaTat99}
Sergiu Klainerman and Daniel Tataru.
\newblock On the optimal local regularity for {Y}ang-{M}ills equations in
  {${\bf R}^{4+1}$}.
\newblock {\em J. Amer. Math. Soc.}, 12(1):93--116, 1999.

\bibitem{KriSchTat09}
J.~Krieger, W.~Schlag, and D.~Tataru.
\newblock Renormalization and blow up for the critical {Y}ang-{M}ills problem.
\newblock {\em Adv. Math.}, 221(5):1445--1521, 2009.

\bibitem{KriTat17}
Joachim Krieger and Daniel Tataru.
\newblock Global well-posedness for the {Y}ang-{M}ills equation in {$4+1$}
  dimensions. {S}mall energy.
\newblock {\em Ann. of Math. (2)}, 185(3):831--893, 2017.

\bibitem{Nai94}
Hisashi Naito.
\newblock Finite time blowing-up for the {Y}ang-{M}ills gradient flow in higher
  dimensions.
\newblock {\em Hokkaido Math. J.}, 23(3):451--464, 1994.

\bibitem{Oh14}
Sung-Jin Oh.
\newblock Gauge choice for the {Y}ang-{M}ills equations using the
  {Y}ang-{M}ills heat flow and local well-posedness in {$H^1$}.
\newblock {\em J. Hyperbolic Differ. Equ.}, 11(1):1--108, 2014.

\bibitem{Oh15}
Sung-Jin Oh.
\newblock Finite energy global well-posedness of the {Y}ang-{M}ills equations
  on {$\Bbb{R}^{1+3}$}: an approach using the {Y}ang-{M}ills heat flow.
\newblock {\em Duke Math. J.}, 164(9):1669--1732, 2015.

\bibitem{OhTat17arx}
Sung-Jin {Oh} and Daniel {Tataru}.
\newblock {The Yang--Mills heat flow and the caloric gauge}.
\newblock {\em arXiv e-prints}, page arXiv:1709.08599, September 2017.

\bibitem{OhTat19}
Sung-Jin Oh and Daniel Tataru.
\newblock The hyperbolic {Y}ang-{M}ills equation for connections in an
  arbitrary topological class.
\newblock {\em Comm. Math. Phys.}, 365(2):685--739, 2019.

\bibitem{OhTat19b}
Sung-Jin Oh and Daniel Tataru.
\newblock The threshold theorem for the {$(4+1)$}-dimensional {Y}ang-{M}ills
  equation: an overview of the proof.
\newblock {\em Bull. Amer. Math. Soc. (N.S.)}, 56(2):171--210, 2019.

\bibitem{OhTat20}
Sung-Jin Oh and Daniel Tataru.
\newblock The hyperbolic {Y}ang-{M}ills equation in the caloric gauge: local
  well-posedness and control of energy-dispersed solutions.
\newblock {\em Pure Appl. Anal.}, 2(2):233--384, 2020.

\bibitem{OhTat21}
Sung-Jin Oh and Daniel Tataru.
\newblock The threshold conjecture for the energy critical hyperbolic
  {Y}ang-{M}ills equation.
\newblock {\em Ann. of Math. (2)}, 194(2):393--473, 2021.

\bibitem{RapRod12}
Pierre Rapha{\"e}l and Igor Rodnianski.
\newblock Stable blow up dynamics for the critical co-rotational wave maps and
  equivariant {Y}ang-{M}ills problems.
\newblock {\em Publ. Math. Inst. Hautes \'Etudes Sci.}, pages 1--122, 2012.

\bibitem{RenRog04}
Michael Renardy and Robert~C. Rogers.
\newblock {\em An introduction to partial differential equations}, volume~13 of
  {\em Texts in Applied Mathematics}.
\newblock Springer-Verlag, New York, second edition, 2004.

\bibitem{SelAch16}
Sigmund Selberg and Achenef Tesfahun.
\newblock Null structure and local well-posedness in the energy class for the
  {Y}ang-{M}ills equations in {L}orenz gauge.
\newblock {\em J. Eur. Math. Soc. (JEMS)}, 18(8):1729--1752, 2016.

\bibitem{ShaTah94}
Jalal Shatah and A.~Shadi Tahvildar-Zadeh.
\newblock On the {C}auchy problem for equivariant wave maps.
\newblock {\em Comm. Pure Appl. Math.}, 47(5):719--754, 1994.

\bibitem{Sim15}
Barry Simon.
\newblock {\em Operator theory}.
\newblock A Comprehensive Course in Analysis, Part 4. American Mathematical
  Society, Providence, RI, 2015.

\bibitem{Tho05}
Gerardus 't~Hooft, editor.
\newblock {\em 50 years of {Y}ang-{M}ills theory}.
\newblock World Scientific Publishing Co. Pte. Ltd., Hackensack, NJ, 2005.

\bibitem{Tao03}
Terence Tao.
\newblock Local well-posedness of the {Y}ang-{M}ills equation in the temporal
  gauge below the energy norm.
\newblock {\em J. Differential Equations}, 189(2):366--382, 2003.

\bibitem{Ach15b}
Achenef Tesfahun.
\newblock Finite energy local well-posedness for the {Y}ang-{M}ills-{H}iggs
  equations in {L}orenz gauge.
\newblock {\em Int. Math. Res. Not. IMRN}, (13):5140--5161, 2015.

\bibitem{Ach15a}
Achenef Tesfahun.
\newblock Local well-posedness of {Y}ang-{M}ills equations in {L}orenz gauge
  below the energy norm.
\newblock {\em NoDEA Nonlinear Differential Equations Appl.}, 22(4):849--875,
  2015.

\bibitem{Wal19}
Alex Waldron.
\newblock Long-time existence for {Y}ang-{M}ills flow.
\newblock {\em Invent. Math.}, 217(3):1069--1147, 2019.

\bibitem{Wei04}
Ben Weinkove.
\newblock Singularity formation in the {Y}ang-{M}ills flow.
\newblock {\em Calc. Var. Partial Differential Equations}, 19(2):211--220,
  2004.

\end{thebibliography}
\bibliographystyle{plain}

\end{document}